\documentclass[11pt]{amsart}
\usepackage{amsmath,amssymb,amsthm}
\usepackage{geometry}
\usepackage{enumitem}
\usepackage{graphicx}
\usepackage{mathrsfs}
\usepackage{float}
\usepackage{color}
\usepackage{thmtools}
\usepackage{thm-restate}
\usepackage{tikz-cd}

\usepackage{xcolor}
\definecolor{darkgreen}{rgb}{0,0.50,0} 
\definecolor{darkred}{rgb}{0.75,0,0}
\definecolor{darkblue}{rgb}{0,0,0.6} 
\definecolor{lightblue}{rgb}{.6,.6,0.9} 
\usepackage[pdfborder=0,pagebackref,colorlinks,citecolor=darkgreen,linkcolor=darkgreen,urlcolor=darkblue]{hyperref}
\renewcommand*{\backref}[1]{}
\renewcommand*{\backrefalt}[4]{({%
    \ifcase #1 Not cited.%
          \or On p.~#2%
          \else On pp.~#2%
    \fi%
    })}

\numberwithin{equation}{section}

\newlength{\storeparskip}
\newcommand{\HH}{H}
\newcommand{\Diff}{\textup{Diff}\,}

\newcommand{\Aut}{\textup{Aut}}
\newcommand{\Fib}{\textup{Fib}}

\newcommand{\aexp}{\mathrm{AExp}}
\newcommand{\Exp}{\mathrm{Exp}}
\newcommand{\Log}{\mathrm{Log}}

\newcommand{\auth}{\mathrm{Aut}(\HH)}
\newcommand{\aut}{\mathrm{Aut}(\HH)}
\newcommand{\vaut}{\mathrm{VAut}(\HH)}
\newcommand{\bal}{\mathrm{Bal}(\HH)}
\newcommand{\vf}{\mathrm{VF}(S^3)}
\newcommand{\taut}{T_{\mathrm{id}}\aut}
\newcommand{\diff}{\diffst}
\newcommand{\diffst}{\mathrm{Diff}(S^3)}

\def\makeautorefname#1#2{\expandafter\def\csname#1autorefname\endcsname{#2}}
\makeautorefname{equation}{Equation}%
\makeautorefname{footnote}{footnote}%
\makeautorefname{item}{item}%
\makeautorefname{figure}{Figure}%
\makeautorefname{table}{Table}%
\makeautorefname{part}{Part}%
\makeautorefname{appendix}{Appendix}%
\makeautorefname{chapter}{Chapter}%
\makeautorefname{section}{Section}%
\makeautorefname{subsection}{Section}%
\makeautorefname{subsubsection}{Section}%
\makeautorefname{paragraph}{Paragraph}%
\makeautorefname{subparagraph}{Paragraph}%
\makeautorefname{theorem}{Theorem}%
\makeautorefname{thm}{Theorem}%
\makeautorefname{addm}{Addendum}%
\makeautorefname{mainthm}{Main theorem}%
\makeautorefname{corollary}{Corollary}%
\makeautorefname{cor}{Corollary}%
\makeautorefname{lemma}{Lemma}%
\makeautorefname{lem}{Lemma}%
\makeautorefname{sublemma}{Sublemma}%
\makeautorefname{sublem}{Sublemma}%
\makeautorefname{subl}{Sublemma}%
\makeautorefname{prop}{Proposition}%
\makeautorefname{property}{Property}
\makeautorefname{pro}{Property}
\makeautorefname{sch}{Scholium}%
\makeautorefname{step}{Step}%
\makeautorefname{conject}{Conjecture}%
\makeautorefname{conj}{Conjecture}%
\makeautorefname{questn}{Question}
\makeautorefname{quest}{Question}
\makeautorefname{qn}{Question}
\makeautorefname{definition}{Definition}%
\makeautorefname{defn}{Definition}%
\makeautorefname{defi}{Definition}%
\makeautorefname{def}{Definition}%
\makeautorefname{dfn}{Definition}%
\makeautorefname{notation}{Notation}
\makeautorefname{notn}{Notation}
\makeautorefname{rem}{Remark}%
\makeautorefname{rems}{Remarks}%
\makeautorefname{rmk}{Remark}%
\makeautorefname{rk}{Remark}%
\makeautorefname{remarks}{Remarks}%
\makeautorefname{rems}{Remarks}%
\makeautorefname{rmks}{Remarks}%
\makeautorefname{rks}{Remarks}%
\makeautorefname{example}{Example}%
\makeautorefname{examp}{Example}%
\makeautorefname{exmp}{Example}%
\makeautorefname{exam}{Example}%
\makeautorefname{exa}{Example}%
\makeautorefname{axiom}{Axiom}%
\makeautorefname{axi}{Axiom}%
\makeautorefname{ax}{Axiom}%
\makeautorefname{case}{Case}%
\makeautorefname{claim}{Claim}%
\makeautorefname{clm}{Claim}%
\makeautorefname{assumpt}{Assumption}%
\makeautorefname{asses}{Assumptions}%
\makeautorefname{conclusion}{Conclusion}%
\makeautorefname{concl}{Conclusion}%
\makeautorefname{conc}{Conclusion}%
\makeautorefname{cond}{Condition}%
\makeautorefname{const}{Construction}%
\makeautorefname{con}{Construction}%
\makeautorefname{criterion}{Criterion}%
\makeautorefname{criter}{Criterion}%
\makeautorefname{crit}{Criterion}%
\makeautorefname{exercise}{Exercise}%
\makeautorefname{exer}{Exercise}%
\makeautorefname{exe}{Exercise}%
\makeautorefname{problem}{Problem}%
\makeautorefname{problm}{Problem}%
\makeautorefname{prob}{Problem}%
\makeautorefname{prob}{Problem}%
\makeautorefname{soln}{Solution}%
\makeautorefname{sol}{Solution}%
\makeautorefname{sum}{Summary}%
\makeautorefname{oper}{Operation}%
\makeautorefname{obs}{Observation}%
\makeautorefname{ob}{Observation}%
\makeautorefname{conv}{Convention}%
\makeautorefname{cvn}{Convention}%
\makeautorefname{warn}{Warning}%
\makeautorefname{note}{Note}%
\makeautorefname{fact}{Fact}%
\makeautorefname{ouch0}{Counterexample}%
\makeautorefname{introthm}{Theorem}%
\newtheorem{theorem}{Theorem}[section]
\newtheorem{introthm}{Theorem}
\newtheorem{cor}{Corollary}[section]
\newtheorem{prop}{Proposition}[section]
\newtheorem{proposition}{Proposition}[section]
\newtheorem{lem}{Lemma}[section]

\theoremstyle{definition}
\newtheorem{defn}{Definition}[section]

\newtheorem{rem}{Remark}[section]
\newtheorem{rems}{Remarks}[section]

\makeatletter
\let\c@cor=\c@thm
\let\c@prop=\c@thm
\let\c@lem=\c@thm
\let\c@conj=\c@thm
\let\c@defn=\c@thm
\let\c@claim=\c@thm
\let\c@quest=\c@thm
\let\c@df=\c@thm
\let\c@exmp=\c@thm
\let\c@rem=\c@thm
\let\c@sch=\c@thm
\let\c@con=\c@thm
\let\c@equation\c@thm
\let\c@introthm=\c@thm
\makeatother

 \numberwithin{equation}{section}

\title[Fibrations of $S^3$ by Simple Closed Curves]{Homotopy Type of the Space of Fibrations\\ of the Three-Sphere by Simple Closed Curves}

\author[DeTurck]{Dennis DeTurck}
\address{Department of Mathematics, The University of Pennsylvania}
\email{deturck@math.upenn.edu}
\author[Fang]{Ziqi Fang}
\address{Department of Mathematics, Massachusetts  Institute of Technology}
\email{ziqifang@mit.edu}
\author[Gluck]{Herman Gluck}
\address{Department of Mathematics, The University of Pennsylvania}
\email{gluck@math.upenn.edu}
\author[Lichtenfelz]{Leandro Lichtenfelz}
\address{Department of Mathematics, Wake Forest University}
\email{lichtel@wfu.edu}
\author[Merling]{\\ Mona Merling}
\address{Department of Mathematics, The University of Pennsylvania}
\email{mmerling@math.upenn.edu}
\author[Wang]{Yi Wang}
\address{Department of Mathematics, University of Illinois Urbana-Champaign}
\email{yiwang20@illinois.edu}
\author[Yang]{Jingye Yang}
\address{Bristol Myers Squibb}
\email{jingyey95@gmail.com}

\makeatletter
\def\l@subsection{\@tocline{2}{0pt}{2pc}{6pc}{}} \makeatother

\date{}

\begin{document}

\newgeometry{margin=0.7in} 
\begin{titlepage}

\stepcounter{page}

\begin{abstract}
We show that the moduli space of all smooth fibrations of a three-sphere by simple
closed curves has the homotopy type of a disjoint union of a pair of two-spheres if the fibers are
oriented, and of a pair of real projective planes if unoriented, the same as for its finite-dimensional
subspace of Hopf fibrations by parallel great circles.
\newline\newline\noindent This moduli space is the quotient of the diffeomorphism group of the three-sphere (a Fréchet Lie group)
by its subgroup of automorphisms of the Hopf fibration, which we show is a smooth Fréchet
submanifold of the diffeomorphism group. Then we show that the moduli space, already known to be a
 Fréchet manifold by \cite{rubinstein}, can be modeled on the concrete Fréchet space of vector fields on the three-sphere which are ``horizontal'' and ``balanced'' with respect to a given Hopf
fibration, and see how the structure of this moduli space helps us to determine its homotopy type.
\end{abstract}

\maketitle

\begin{figure}[H]
    \centering
    \includegraphics[scale=0.8]{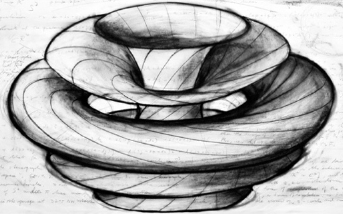}
    \caption{\centering Hopf fibration of the $3$-sphere by parallel great circles. \\ Lun-Yi Tsai, Charcoal and graphite on paper, 2007.}
    \label{fig:HopfTsai}
\end{figure}

\begingroup%
\setlength{\parskip}{\storeparskip}
\tableofcontents
\endgroup%

\end{titlepage}
\restoregeometry 

\section{Introduction}

We begin with the Hopf fibration $\HH \colon S^1 \subset S^3 \xrightarrow{p} S^2$, in which the unit 3-sphere $S^3$ in real 4-space $\mathbb{R}^4$ is filled with parallel great circles, as discovered by Heinz Hopf in 1931. These great circles are the intersections of $S^3$ with the complex lines in the $\mathbb{C}^2$ corresponding to an orthogonal complex structure on $\mathbb{R}^4$ .The
projection map $S^3 \xrightarrow{p} S^2$ was the first example of a homotopically nontrivial map from a sphere
to another sphere of lower dimension, and signaled the birth of homotopy theory.


How many Hopf fibrations of $S^3$ are there? If we orient all the great circle fibers and then focus
on the direction of the fiber through a single point of $S^3$, we can see a two-sphere's worth of Hopf
fibrations with a right-handed screw sense, and another two-sphere's worth with a left-handed
screw sense. There are no more.

In 1983,  \cite{gluck_warner} showed that the much larger (infinite-dimensional) space
of all (topological or smooth) fibrations of $S^3$ by (not necessarily parallel) oriented great circles
deformation retracts to its subspace of Hopf fibrations, and so also has the homotopy type of a pair
of two-spheres. In this paper, "smooth" always means of class $C^\infty$.

Carrying that result to its extreme, we show here that the same homotopy type persists when we
expand our view to include all possible smooth fibrations of $S^3$ by simple closed curves.

\begin{restatable}{introthm}{mainthm}
\label{mainthm}
The space of all smooth fibrations of a three-sphere by simple closed curves, with a $C^{\infty}$ topology, has the homotopy type of a disjoint union of a pair of two-spheres if the fibers are oriented, and of a pair of real projective planes if unoriented, the same as for its finite-dimensional subspace of Hopf fibrations by parallel great circles.
\end{restatable}

We position our story in the world of (typically infinite-dimensional) metrizable Fr\'echet manifolds and smooth maps between them, give a brief introduction to this in \autoref{preliminaries}, and profit later because in this world, weak homotopy equivalences can be promoted to homotopy equivalences. In particular, the group $\diffst$ of all diffeomorphisms of the three-sphere lives in this world, and is a Fr\'echet Lie group.

Inside that group is the subgroup $\aut$ of automorphisms of the Hopf fibration $\HH$, meaning diffeomorphisms of $S^3$ which permute the fibers of $\HH$, with stretching and  compressing allowed.

To confirm that this subgroup lives in the Fr\'echet world, we prove

\begin{introthm}
\label{submanifold_thm}
The subgroup $\aut$ of automorphisms of a Hopf fibration $\HH$ is a smooth Fr\'echet submanifold of $\diffst$.
\end{introthm}

Fr\'echet manifolds are modeled after Fr\'echet (vector) spaces, which can then be viewed as their tangent spaces at distinguished points. For example, the tangent space to $\diffst$ at the identity is the Fr\'echet space of all smooth vector fields on the three-sphere. We can write it as
\begin{equation*}
T_{\mathrm{id}}\diffst = \mathrm{VF}(S^3)
\end{equation*}
and give it the $L^2$ inner product defined by the integral of the pointwise inner product,
\begin{equation*}
\langle X, Y \rangle = \int_{S^3} X(x) \cdot Y(x)\, \mathrm{d(Vol)}.
\end{equation*}

\newpage

The \emph{Riemannian exponential map} takes a portion of the tangent space
$T_{\mathrm{id}}\,\diffst = \mathrm{VF}(S^3)$ to a portion of the diffeomorphism group $\diffst$, and provides
coordinate charts. A useful variant of this, the \emph{aligned Riemannian exponential map}, is defined
in \autoref{submanifoldsection}, proved there to give the same Fr\'echet structure on $\diffst$, and then used to good effect in the proof of \autoref{submanifold_thm} above.

Likewise, the tangent space to the subgroup $\aut$ at the identity consists of smooth vector fields on the three-sphere which project down by $p_*$ to well-defined vector fields on the base space $S^2$ of the Hopf fibration. This means that their ``vertical'' components -- those parallel
to the Hopf fibers -- are unconstrained, but that their ``horizontal'' components orthogonal to the Hopf fibers point systematically from each Hopf fiber to another Hopf fiber. We make this precise in \autoref{preliminaries}. Because of this alternative description, we sometimes write the tangent space to $\aut$ at the identity as
\begin{equation*}
\taut = \mathrm{PVF}(S^3)
\end{equation*}
which we refer to as the space of \emph{projectable vector fields}.

The $L^2$-orthogonal complement to $\taut$ consists of smooth vector fields $V$ on $S^3$ which
are horizontal and which integrate to zero along each Hopf fiber $F$. We refer to these vector fields
as \emph{balanced}, denote the subspace of them by $\mathrm{BVF}(S^3)$, and prove that the direct sum
\begin{equation*}
\mathrm{PVF}(S^3) \oplus \mathrm{BVF}(S^3) = \mathrm{VF}(S^3)
\end{equation*}
gives all of $\mathrm{VF}(S^3)$. Applying the standard Riemannian exponential map to a neighborhood of the zero field in $\mathrm{BVF}(S^3)$ gives us a local Fr\'echet submanifold $\bal$ of $\diffst$, our balanced diffeomorphisms of $S^3$.


We denote set of all smooth fibrations of $S^3$ by oriented simple closed curves by $\mathrm{Fib}(S^3)$. Because all of these fibrations are equivalent
to one another by diffeomorphisms of $S^3$, we have a sequence
\begin{equation*}
\aut \hookrightarrow \diffst \to \mathrm{Fib}(S^3).
\end{equation*}

Both $\aut$ and $\diffst$ with their $C^\infty$ topologies are Fr\'echet manifolds, and as a start for
$\mathrm{Fib}(S^3)$ to join this world, we give it the quotient topology from the above projection. We regard the local Fr\'echet submanifold $\bal$ of $\diffst$ as the most natural home for a smooth slice to prove the smooth Fr\'echet bundle structure, and devote \autoref{slice_section} to providing this:

\begin{introthm}
\label{slice_thm}
The sequence $\aut \hookrightarrow \diffst \to \mathrm{Fib}(S^3)$ admits a smooth slice consisting of balanced diffeomorphisms, and is therefore a smooth fiber bundle of Fr\'echet manifolds.
\end{introthm}

The Fr\'echet bundle structure was already proved in \cite{rubinstein} by a  different argument based on the center-of-mass construction of \cite{Kar}. Instead, here, the existence of a smooth slice consisting of balanced diffeomorphisms is confirmed by a convex optimization
argument, and has the advantage of displaying the Fr\'echet manifold structure of our moduli space $\mathrm{Fib}(S^3)$ as modeled on the ``concrete'' Fr\'echet space of balanced vector fields on $S^3$.

With all this in hand, the proof of \autoref{mainthm} focuses on the Fr\'echet fiber bundle from \autoref{slice_thm}. By a preliminary argument, we determine the homotopy groups of the fiber $\aut$.
And we know from Hatcher's theorem~\cite{hatcher_smale} that $\diffst$ deformation retracts to the
orthogonal group $O(4)$. We then use the rigid symmetries of the Hopf fibration~\cite{gluckwarnerziller} to identify a finite-dimensional sub-bundle of the above bundle with well-known fiber and total space and with base space $S^2 \sqcup S^2$.
Its inclusion into the big bundle above is seen to induce isomorphisms of homotopy groups of
fiber and of total space, so by the Five Lemma, it induces isomorphisms of homotopy groups of
base spaces. Then the weak homotopy equivalence $S^2 \cup S^2 \to \Fib(S^3)$ is promoted to a
homotopy equivalence using the Fr\'echet machinery, proving \autoref{mainthm}.

Then, thanks to \cite{anderson1966} and \cite{henderson_schori_1970}, we can bring our story to a close by promoting \autoref{mainthm} to the following result.

\begin{restatable}{introthm}{bettermainthm}
\label{better_main_thm}
The space $\mathrm{Fib}(S^3)$ is homeomorphic to two disjoint copies of $S^2 \times \mathbb{R}^\infty$.
\end{restatable}

\noindent\textbf{Closely related results.}  \autoref{mainthm} fits into a general program of studying the moduli space of smooth fibrations of a given manifold $M$
by copies of a given fiber $F$ and determining not only its number of components, but its entire homotopy type. For most closed aspherical orientable $3$-manifolds $E$, the space of smooth fibrations of $E$ by simple closed curves is known to have contractible components~\cite{rubinstein}.

Some of the authors of this paper have also pursued other related results of this form.
For example, in \cite{wangyang} the authors show that the space of all smooth fibrations of $S^1\times S^2$ by oriented simple closed curves has four components, each having the homotopy type of the based loop space $\Omega S^3$ of the 3-sphere; this is the first known example of a space of smooth fibrations of a 3-manifold which does
not have the homotopy type of a finite-dimensional space.
As another example, in a more systematic spirit, \cite{ziqi} builds on a range of existing work to establish general results and frameworks for studying such infinite-dimensional moduli spaces, from both topological and differential perspectives. In particular, applications include deducing finite-dimensional deformation retracts inside these infinite-dimensional moduli of smooth fibrations for various low-dimensional ambient manifolds (such as closed surfaces and irreducible closed 3-manifolds, recovering the known cases as well as completing the remaining unknown cases for dimensions up to 3). For example, the author shows that the space of all smooth fibrations of a flat two-torus by oriented simple closed curves has contractible components, and each component contains a unique fibering by oriented parallel closed geodesics, to which the component deformation retracts via the curve-shortening flow for families.

To see \autoref{submanifold_thm} in a wider context, we may ask, for other natural geometric structures on the three-sphere, whether the corresponding automorphism subgroups of $\diffst$ also inherit a smooth Fr\'echet submanifold structure.

This is known to hold for the subgroup $\mathrm{Aut}_1(\xi)$ of $\diffst$ consisting of strict contactomorphisms~\cite{omori}. For the full contactomorphism subgroup $\mathrm{Aut}(\xi)$, this appeared in~\cite{Ham}
as Problem 2.5.6, and to our knowledge is still open. By contrast, in the Sobolev category,
$\mathrm{Aut}(\xi)$ is \emph{not} a smooth Fr\'echet submanifold of $\diffst$ (\cite{EP, omori, Smo}).

\noindent\textbf{Conventions.}
Throughout, $\Diff(S^3)$ denotes the topological group of diffeomorphisms of the $\text{3-sphere}$ with
the $C^\infty$ topology. Multiplication and inversion are continuous in this topology, and later, when
we promote $\Diff(S^3)$ to the category of smooth Fr\'echet  manifolds and smooth maps between them, we will note that multiplication and inversion are smooth, and that $\Diff(S^3)$ is a Fr\'echet Lie group.

When we speak of smooth fibrations of a three-sphere by simple closed curves, we mean a fiber bundle in the category of smooth manifolds and smooth maps between them. These are all equivalent in the topological category, and we will indicate later why they are all equivalent in the smooth category.

\noindent\textbf{Acknowledgements.}
We are grateful for the many contributions to this project arising from conversations with Tom Goodwillie, Kiyoshi Igusa, Sander Kupers, Rob Kusner, Cary Malkiewich and Jim Stasheff. Merling acknowledges partial support from NSF DMS grants CAREER 1943925 and FRG 2052988.  Wang acknowledges partial support from NSF GRFP 1650114.

\newpage





\section{Preliminaries}
\label{preliminaries}

In this section, we recall some preliminary results about Fréchet manifolds and their homotopy theory.  In addition, we give explicit descriptions, in terms of basic left-invariant vector fields on $S^3$, of the various Fréchet spaces of interest to us.

\subsection{Fréchet vector spaces and manifolds}

Fréchet spaces, manifolds and Lie groups provide the setting for extending the theory of finite-dimensional smooth manifolds and smooth maps between them to infinite dimensions.  We set vocabulary here and then refer the reader to some of the major sources.

A \emph{Fréchet space $V$} is a complete metrizable vector space whose topology is induced by a countable family of semi-norms, where a semi-norm $r$ is like a norm, except for the axiom that asks for $r(v) = 0$ to imply that $v = 0$.

A collection $\{ \| \cdot \|_\alpha : \alpha \in I \}$ of semi-norms on $V$ defines a unique topology if we declare that a sequence $f_j$ converges to $f$ if and only if $\|f_j - f\|_\alpha \to 0$ for each $\alpha \in I$.

A Fréchet space $V$ is \emph{Hausdorff} if and only if $\| f \|_\alpha = 0$ for all $\alpha$ implies $f = 0$.  It is \emph{metrizable} if and only if the collection $\{ \| \cdot \|_\alpha : \alpha \in I \}$ can be replaced by a countable collection $\{ \| \cdot \|_n : n \in \mathbb{N} \}$, which defines the same topology. In that case, one such metric is
\begin{equation*}
d(x, y) = \sum_{n=1}^\infty 2^{-n} \frac{ \| x - y \|_n }{ 1 + \| x - y \|_n }.
\end{equation*}
It is \emph{complete} if every Cauchy sequence converges, where ``Cauchy'' here means that for each fixed $\alpha \in I$, we have
\begin{equation*}
\| f_i - f_j \|_\alpha \to 0 \quad \text{as} \quad i, j \to \infty.
\end{equation*}

Smooth ($C^\infty$) maps between open subsets of Fréchet spaces are defined, as in the finite-dimensional case, in terms of the convergence of various difference quotients.

\begin{defn}
A \emph{Fréchet manifold} is a Hausdorff topological space with an atlas of coordinate charts taking their values in Fréchet spaces, such that the coordinate transition functions are all smooth maps between open sets in Fréchet spaces.
\end{defn}

We collect the following facts about Fréchet manifolds from the paper~\cite{palais}.

\begin{proposition}[{\cite[Cor. 2]{palais}}]\label{component_metrizable}
If each component of a Fréchet manifold $X$ satisfies the second axiom of countability, then $X$ is metrizable.
\end{proposition}

A space $X$ is said to be \emph{dominated} by a space $Y$ if there exist maps $f : X \to Y$ and $g : Y \to X$ such that $g \circ f : X \to X$ is homotopic to the identity map of $X$.

\begin{proposition}[{\cite[Thm. 14]{palais}}]
A metrizable Fréchet manifold is dominated by a simplicial complex.
\end{proposition}

\begin{proposition}[{\cite[Thm. 15]{palais}}]\label{homotopy_upgrade}
Let $X$ and $Y$ be metrizable Fréchet manifolds.  Then any weak homotopy equivalence $f : X \to Y$ is in fact a homotopy equivalence.
\end{proposition}

We will use this in the last step of the proof of our  \autoref{mainthm} to promote weak homotopy equivalence to homotopy equivalence.

A \emph{Fréchet Lie group} is a Fréchet manifold $G$ with a group structure in which the multiplication map $G \times G \to G$ and the inversion map $G \to G$ are both smooth maps of Fréchet manifolds. For example, if $M$ is a finite-dimensional smooth manifold, then multiplication (i.e., composition) and inversion in $\mathrm{Diff}(M)$ are smooth maps, making it into a Fréchet Lie group.

The example of most importance to us is the space $\diffst$ of smooth diffeomorphisms of the 3-sphere to itself, which is modelled on the Fréchet vector space $\mathrm{VF}(S^3)$ of smooth vector fields on $S^3$. The coordinate charts are provided by the \emph{Riemannian exponential map}, which converts each short vector field $V$ into a diffeomorphism sending points $x$ of $S^3$ to the endpoints of the geodesics directed by $V(x)$. The following result is of importance to us.

\begin{theorem}[{\cite{KM}}]\label{separable_metrizable}
For any closed smooth finite-dimensional manifold $M$, the Fréchet Lie group $\mathrm{Diff}(M)$ of smooth diffeomorphisms is separable and metrizable.
\end{theorem}

For a detailed discussion of Fréchet vector spaces and manifolds, we refer the reader to the paper~\cite{Ham} and the book~\cite{KM}. To get a sense of the early development of this subject, see~\cite{Eells58, Eells, palais, leslie}. We have also offered an overview in~\cite[Appendix~A]{deturck2024deformation}.

\subsection{Left-invariant vector fields on \texorpdfstring{$S^3$}{S3}}

Vector fields on the 3-sphere $S^3$ become more concrete when we view $S^3$ as the space of unit quaternions, and use the three orthonormal left-invariant vector fields 
\begin{equation*}
A(x) = x i , \quad B(x) = x j , \quad C(x) = x k ,
\end{equation*}
written in terms of quaternion multiplication, as a basis for $\mathrm{VF}(S^3)$ as a module over the ring of smooth real-valued functions on $S^3$.

\begin{figure}[H]
    \centering
    \includegraphics[scale=0.23, angle=270]{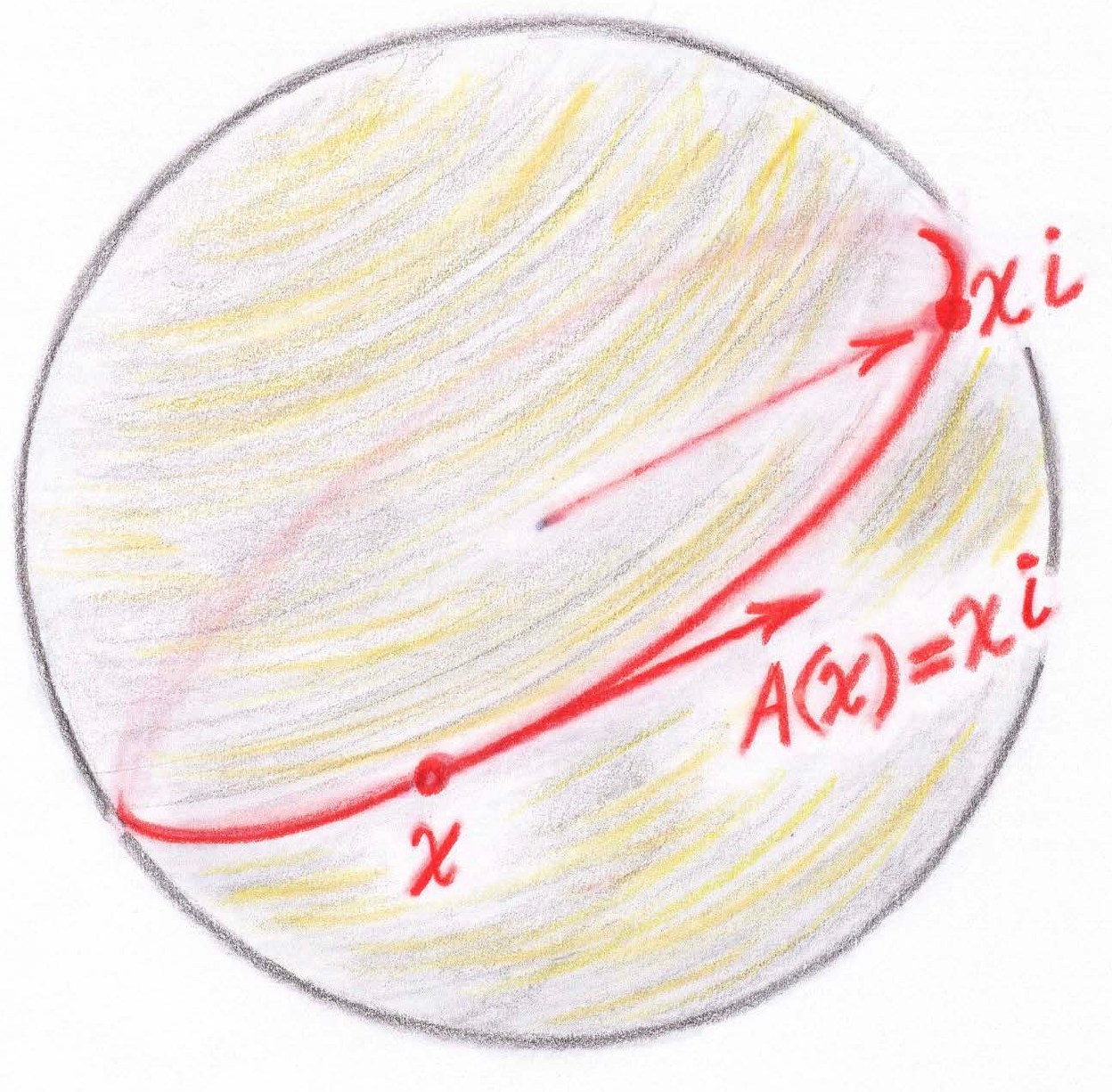}
    \caption{\centering The quaternion product  $A(x) = xi$  may be visualized as a point on $S^3$,
    or as a unit vector in $\mathbb{R}^4$ from the origin to this point, \\ or as a unit vector tangent at $x$ to the great circle from $x$ to $xi$.}
    \label{fig:fig1}
\end{figure}

Any smooth vector field $X$ on $S^3$ can be expressed as $X = f\,A + g\,B + h\,C$, where the coefficients are smooth real-valued functions on $S^3$.  We will need to know the Lie brackets of these basis vector fields, which are easily computed to be
\begin{equation}\label{lie_bracket_abc}
[A, B] = 2 C , \quad [B, C] = 2 A , \quad [C, A] = 2 B .
\end{equation}

We next choose our Hopf fibration $\HH$ so that $A$ is the unit vector field tangent to the Hopf fibers, which are then the great circles in $S^3$ taken to themselves via right multiplication by $i$, and oriented in this direction from $x$ to $x i$.  We write this Hopf fibration as 
\begin{equation*}
\HH \; : \; S^1 \;\subset\; S^3 \;\overset{p}{\to}\; S^2.
\end{equation*}

The Fréchet vector space $\mathrm{VF}(S^3)$ is the tangent space at the identity to the Fréchet Lie group $\diffst$.  We will use the $L^2$ inner product on $\mathrm{VF}(S^3)$, which is given by the integral of the pointwise inner product on $S^3$:
\begin{equation}\label{l2_metric_def}
\langle X, Y \rangle = \int_{S^3} X(x) \cdot Y(x) \; \mathrm{d(Vol)}.
\end{equation}

\begin{rem}
The metric topology coming from this inner product is not the same as the smooth topology on $\mathrm{VF}(S^3)$; it does not have nearly enough open sets.  And unlike the smooth topology, this metric topology is not complete because differentiability can be lost in the limit.
\end{rem}

We now introduce the main player, the automorphism group of the Hopf fibration.

\begin{defn}
Let $\auth$ denote the subgroup of $\diffst$ consisting of diffeomorphisms which permute the fibers of $\HH$, with stretching and compressing allowed.
\end{defn}

We will see in the proof of Theorem B that $\auth$ is a smooth Fréchet submanifold of $\diffst$, and give in advance a characterization of its tangent space at the identity.

\begin{proposition}\label{prop_taut_description}
The smooth vector field $X = f A + g B + h C$ on $S^3$ lies in the tangent space $T_{\mathrm{id}}\auth$ if and only if
\begin{enumerate}
    \item $f = $ any smooth real-valued function on $S^3$; \smallskip
    \item $g = - \frac{1}{2} A h$, indicating a directional derivative, here and below; \smallskip
    \item $h = \frac{1}{2} A g$.
\end{enumerate}
\end{proposition}

\begin{proof}
To begin, we assert that
\begin{equation*}
T_{\mathrm{id}}\auth = \{ X \in \mathrm{VF}(S^3) : \mathcal{L}_X A = \lambda A \text{ for smooth } \lambda : S^3 \to \mathbb{R} \} ,
\end{equation*}
where $\mathcal{L}_X$ is the Lie derivative with respect to $X$.

To see this, let $\{\varphi_t\}$ be a curve of diffeomorphisms of $S^3$ through the identity when $t = 0$, all lying in the subgroup $\auth$ of $\diffst$, and let $X \in \mathrm{VF}(S^3)$ be tangent to this curve at the identity. Then by definition,
\begin{equation*}
(\mathcal{L}_X A)(x) = \left. \frac{d}{dt} \right|_{t=0} \varphi_t^* A(\varphi_t(x)) .
\end{equation*}
In order for $\varphi_t \in \auth$, the pullback $\varphi_t^* A(\varphi_t(x))$ must be some variable multiple of $A(x)$, depending on both $x$ and $t$, and therefore its time-zero derivative, $(\mathcal{L}_X A)(x)$ must be a variable multiple of $A(x)$, depending just on $x$. That is, $(\mathcal{L}_X A)(x) = \lambda(x) A(x)$.

Now write a vector field $X$ on $S^3$ as $X = f A$ $+$ $g B$ $+$ $h C$ and compute $\mathcal{L}_X A$ to see what constraints the condition $\mathcal{L}_X A = \lambda A$ imposes on the coefficients $f, g$ and $h$.

Notationally, we switch from Lie derivatives to Lie brackets and compute
\begin{align*}
\mathcal{L}_X A &= [X, A] = [fA + gB + hC, A] \\
&= [fA, A] + [gB, A] + [hC, A] \\
&= -[A, fA] - [A, gB] - [A, hC] \\
&= - (A f) A - f[A, A] - (A g) B - g[A, B] - (A h) C - h[A,C] \\
&= - (A f) A - (A g) B - g (2 C) - (A h) C - h ( - 2 B ) \\
&= - (A f) A + (2 h - A g) B - (2 g + A h) C .
\end{align*}
Therefore $X \in T_{\mathrm{id}}\auth$ if and only if $-A f = \lambda$ for some smooth real-valued function $\lambda$, which is no constraint at all, and $2 h - A g = 0$ and $2 g + A h = 0$. This completes the proof of the proposition.
\end{proof}

\begin{rems} We collect the following observations.
\begin{itemize}
    \item The vector fields in $\taut$ are said to be \emph{projectable} or \emph{aligned} with the Hopf projection $p \colon S^3 \to S^2$ in the sense that they project down to well-defined vector fields on the base space $S^2$. The space of such fields is also denoted $\mathrm{PVF}(S^3)$.
    \item We also note that the above conditions persist under limits in the smooth topology on $\mathrm{VF}(S^3)$,
    confirming that $\taut$ is a closed subspace.
    \item The tangent space $\taut$ is a module over the ring of smooth real-valued functions on
    the base space $S^2$ of the Hopf fibration $\HH$, equivalently over the ring of smooth real-valued
    functions on the total space $S^3$ which are constant along each Hopf fiber.
    \item We can decouple equations $(2)$ and $(3)$ in  \autoref{prop_taut_description} by differentiating again with respect 
    to $A$ and, paying attention to constants, learn that

    \begin{gather}\label{taut_description_2}
    \begin{split}
        \text{(i)}\quad & g(t) = c_1 \cos 2t + c_2 \sin 2t \\
        \text{(ii)}\quad & h(t) = c_2 \cos 2t - c_1 \sin 2t ,
    \end{split}
    \end{gather}
where the Hopf circles are parametrized by $0 \leq t \leq 2\pi$, and where the coefficients $c_1$ and $c_2$ are constant along each Hopf circle, but otherwise vary smoothly. The system of equations in \eqref{taut_description_2} persists (with a change of constants) if we replace $t$ by $t + a$, since there is no distinguished point on a Hopf circle at which to place the coordinate $t = 0$.
\end{itemize}

\end{rems}

\begin{figure}[H]
    \centering
    \includegraphics[scale=0.25]{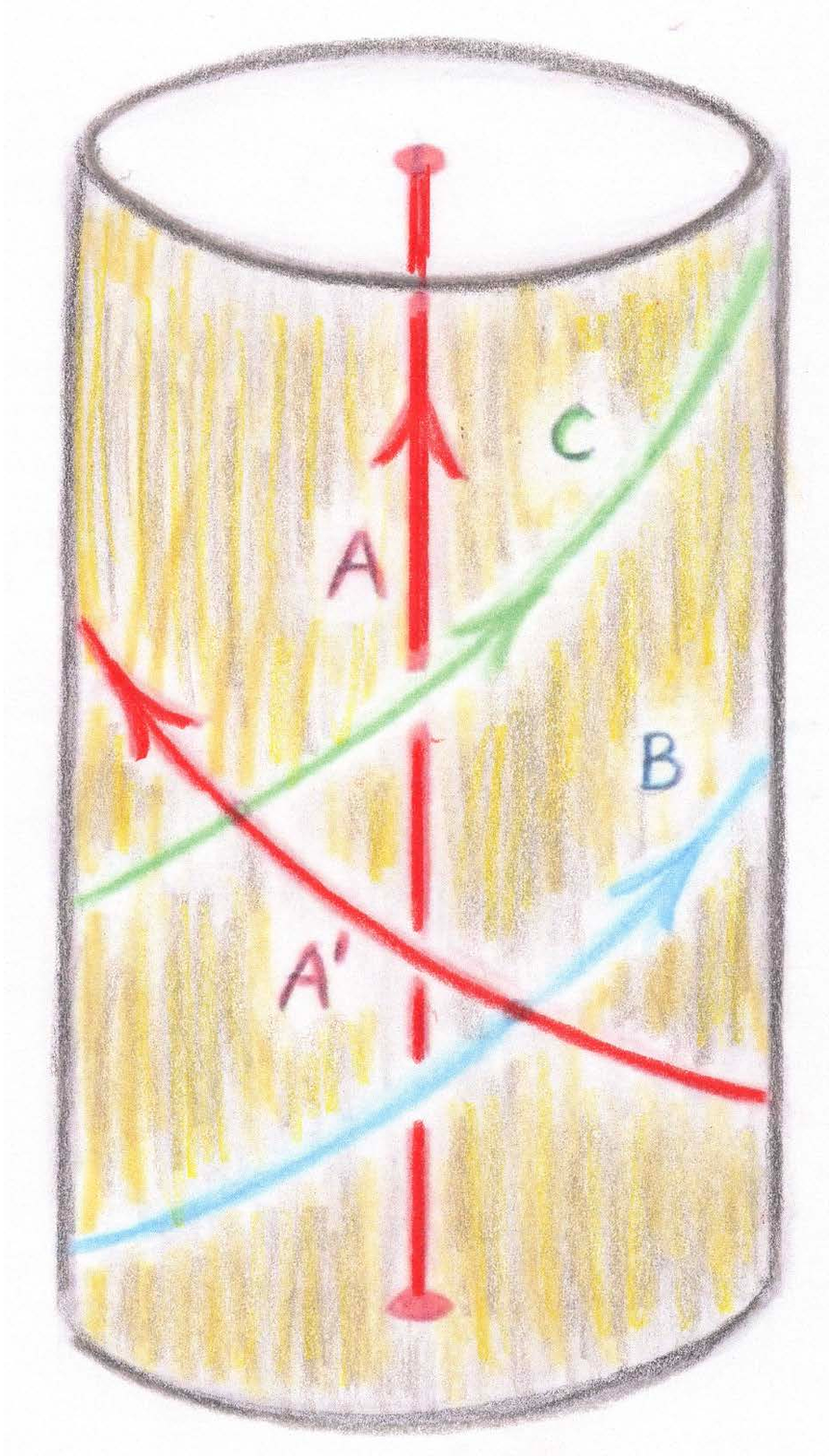}
    \caption{\centering A tubular neighborhood of a Hopf fiber. The curves $A$ and $A'$ are great circle fibers of $\HH$, while the curves $B$ and $C$ are great circles representing  exponentiated Jacobi fields along $A$.}
    \label{fig:fig2}
\end{figure}

In  \autoref{fig:fig2}, we show a thin tubular neighborhood of one of the fibers of our Hopf fibration $\HH$, with the neighborhood cut open and straightened for easier viewing. The orange-colored central fiber is labeled $A$ because it is an integral curve (great circle) of the left-invariant vector field $A$ on $S^3$. It is oriented in the direction of right multiplication by the quaternion $i$. On the boundary of this tubular neighborhood we show in orange another oriented fiber of $\HH$, labeled $A'$.

We sort out left- and right-handed screw sense as follows. Let us orient $\mathbb{R}^4$ via the ordered basis $1, i, j, k$, and then orient $S^3$ so that at each point $x \in S^3$, the unit vector $x$ followed by the orientation of the tangent plane $T_x S^3$ agrees with the orientation of $\mathbb{R}^4$. Thus at the point $1 \in S^3$, the vectors $i, j, k$ give an ordered basis for $T_1 S^3$, while at the point $x \in S^3$, the vectors $A(x) = x i$, $B(x) = x j$, $C(x) = x k$ give an ordered basis for $T_xS^3$.

Now the unit circle in the $(1, i)$-plane is a Hopf fiber oriented from $1$ towards $i$, while the unit circle in the $(j, k)$-plane is a Hopf fiber oriented from $j$ to $j i = -k$. Thus the oriented planes spanned by these two fibers have intersection number $-1$, and hence these two oriented fibers of $\HH$ have linking number $-1$. By convention, we say that this Hopf fibration $\HH$ has a \emph{left-handed screw sense}, as shown in the figure.

We also indicate in blue how the vector field $B(x)$ along the central orange fiber of $\HH$ turns as the point $x$ moves along this fiber in the direction of its orientation. The blue curve can be thought of as the exponential image of the vector field $\epsilon B(x)$ along this central fiber. It is drawn to depict a right-handed screw sense because $L_A B = 2 C$ and $L_A C = -2 B$. And likewise for the vector field $C(x)$, whose exponential image along the central fiber of $\HH$ is shown in green.

A straightforward computation shows that the blue and green curves are fibers of a Hopf fibration $\HH'$ whose orbits are the great circles in $S^3$ taken to themselves by left multiplication by $i$, a kind of sinister first cousin to $\HH$, and so $\HH'$ has a right-handed screw sense. It follows that the vector fields $B(x)$ and $C(x)$ are Jacobi fields along the central orange fiber of $\HH$.

Next, we introduce the subgroup of $\auth$ consisting of those automorphisms that take each Hopf fiber to itself.

\begin{defn}
Let $\vaut \subset \auth$ be the subgroup of automorphisms of $\HH$ which take each oriented fiber to itself, but still with stretching and compressing allowed.
\end{defn}

Note that the following condition, which ensures that a smooth vector field lies in the tangent space to $\vaut$ at the identity, is immediate from the definition.

\begin{proposition}
The smooth vector field $X = fA + gB + hC$ lies in $T_{\mathrm{id}}\vaut$ if and only if
\begin{itemize}
    \item[(1)] $f =$ any smooth real-valued function on $S^3$; \smallskip 
    \item[(2)] $g = 0$; \smallskip 
    \item[(3)] $h = 0$.
\end{itemize}
\end{proposition}

In a similar way, we give a characterization of smooth vector fields which live in the orthogonal complement of the tangent space to $\auth$ at the identity.

\begin{proposition}\label{taut_perp_description}
The $L^2$ orthogonal complement $\taut^{\perp}$ consists of vector fields of the form ${X = fA + gB + hC}$ such that \smallskip 
\begin{itemize}
    \item[(1)] $f = 0$, and along each Hopf circle, we have that  
    \item[(2)] $\int\limits_0^{2\pi} g \cos 2t - h \sin 2t\, dt = 0$, 
    \item[(3)] $\int\limits_0^{2\pi} g \sin 2t + h \cos 2t\, dt = 0$.
\end{itemize}
\end{proposition}

\begin{proof}
For $X = fA + gB + hC$ to be $L^2$ orthogonal to $T_{\mathrm{id}}\auth$, we certainly must have $f = 0$. Then, referring to the decoupled equations \eqref{taut_description_2}, we must also have along each Hopf fiber that
\begin{equation}
\int\limits_0^{2\pi} g (c_1 \cos 2t + c_2 \sin 2t) + h (c_2 \cos 2t - c_1 \sin 2t)\, dt = 0,
\end{equation}
where $c_1$ and $c_2$ are smooth functions on $S^3$ which are constant on Hopf fibers. Putting $c_1 = 1$ and $c_2 = 0$ yields condition $(2)$ above, while putting $c_1 = 0$ and $c_2 = 1$ gives condition $(3)$.

Alternatively expressed, the vector field $X$ is $L^2$-orthogonal to $T_{\mathrm{id}}\auth$ if and only if it is horizontal and is $L^2$-orthogonal along each Hopf circle to the two-dimensional linear space of vector fields there discussed above.
\end{proof}

Just as we noted earlier for $T_{\mathrm{id}}\auth$, its $L^2$ orthogonal complement $T_{\mathrm{id}}\auth^\perp$ is a module over the ring of smooth functions on $S^2$, equivalently over the ring of smooth real-valued 
functions on the total space $S^3$ which are constant along each Hopf fiber, and this lets us 
convert $L^2$ orthogonality on $S^3$ to $L^2$ orthogonality along each great circle fiber of $\HH$.

\begin{proposition}\label{prop_l2_direct_sum}
We have the $L^2$ orthogonal direct sum
\begin{equation}
T_{\mathrm{id}}\auth \oplus T_{\mathrm{id}}\auth^\perp = T_{\mathrm{id}}\diffst = \mathrm{VF}(S^3).
\end{equation}
\end{proposition}

\begin{proof}
This would have been immediate if $T_{\mathrm{id}}\diffst$, with its $L^2$ inner product, were complete,
and thus a Hilbert space. Given a smooth vector field $X = fA + gB + hC$ on $S^3$, we want 
to write $X = X_1 + X_2$,
\begin{equation}
X_1 = f_1 A + g_1 B + h_1 C \in T_{\mathrm{id}}\auth \quad \text{and} \quad X_2 = f_2 A + g_2 B + h_2 C \in T_{\mathrm{id}}\auth^\perp.
\end{equation}
To begin, we must set $f_1 = f$ and $f_2 = 0$. We revisit the decoupled equations for vector fields in $T_{\mathrm{id}}\auth$ given in \eqref{taut_description_2}:
\begin{gather}
\begin{split}
\text{(i)}\quad g(t) &= c_1 \cos 2t + c_2 \sin 2t, \\
\text{(ii)}\quad h(t) &= c_2 \cos 2t - c_1 \sin 2t .
\end{split}
\end{gather}
Then consider the vector fields
\begin{equation}
Y = 0A + (\cos 2t)B - (\sin 2t)C,
\end{equation}
\begin{equation}
Z = 0A + (\sin 2t)B + (\cos 2t)C,
\end{equation}
obtained from $(i)$ and $(ii)$ above by first setting $c_1 = 1$ and $c_2 = 0$ to get $Y$, and then setting $c_1 = 0$ and $c_2 = 1$ to get $Z$. They form a ``basis'' for the horizontal vector fields in $\taut$ when this vector space is regarded as a module over the ring of smooth functions on $S^2$, and when such functions are lifted to smooth functions on $S^3$ constant on Hopf fibers.

In the $L^2$ inner product along a single Hopf fiber, we have
\begin{equation}
\langle Y, Y \rangle = \int\limits_0^{2\pi} \cos^2 2t + \sin^2 2t\, dt = 2\pi ,
\end{equation}
so in this sense, $Y / \sqrt{2\pi}$ is a unit vector, and likewise for $Z$. Furthermore $\langle Y, Z \rangle = 0$.

Now, given the vector field $X = fA + gB + hC$ on $S^3$, we work along a single Hopf fiber
and orthogonally project $X$ to $X_1 \in T_{\mathrm{id}}\auth$, writing
\begin{equation}
X_1 = fA + \frac{\langle X, Y \rangle}{2\pi} Y + \frac{\langle X, Z \rangle}{2\pi} Z .
\end{equation}
Here
\begin{gather}
\begin{split}
\langle X, Y \rangle &= \int\limits_0^{2\pi} g \cos 2t - h \sin 2t\, dt, \qquad \text{and} \qquad
\langle X, Z \rangle = \int\limits_0^{2\pi} g \sin 2t + h \cos 2t\, dt,
\end{split}
\end{gather}
which we abbreviate by $\alpha$ and $\beta$, respectively. Then
\begin{equation}
\begin{split}
X_1 &= fA + \frac{\alpha}{2\pi} (\cos 2t\, B - \sin 2t\, C) + \frac{\beta}{2\pi} (\sin 2t\, B + \cos 2t\, C) \\
&= fA + \left( \frac{\alpha}{2\pi} \cos 2t + \frac{\beta}{2\pi} \sin 2t \right) B + \left( -\frac{\alpha}{2\pi} \sin 2t + \frac{\beta}{2\pi} \cos 2t \right) C ,
\end{split}
\end{equation}
which lies in $T_{\mathrm{id}}\auth$. So the task now is to show that the difference,
\begin{equation}
\begin{split}
X_2 = X - X_1 = 0A + \left(g - \frac{\alpha}{2\pi} \cos 2t - \frac{\beta}{2\pi} \sin 2t\right) B \\
+ \left(h + \frac{\alpha}{2\pi} \sin 2t - \frac{\beta}{2\pi} \cos 2t\right) C,
\end{split}
\end{equation}
lies in $T_{\mathrm{id}}\auth^\perp$. In this last formula, let $g_2$ and $h_2$ abbreviate the coefficients of $B$ and $C$, respectively. So we must confirm the conditions for membership in $T_{\mathrm{id}}\auth^\perp$ as they were given in  \autoref{taut_perp_description}:
\begin{gather}\label{eq_ort_relation}
\begin{split}
\int\limits_0^{2\pi} g_2 \cos 2t - h_2 \sin 2t\, dt = 0, \qquad \text{and} \qquad 
\int\limits_0^{2\pi} g_2 \sin 2t + h_2 \cos 2t\, dt = 0.
\end{split}
\end{gather}
We prove the first equation in \eqref{eq_ort_relation}.
\begin{gather*}
\begin{split}
\int\limits_0^{2\pi} g_2 \cos 2t - h_2 \sin 2t\, dt  
&= \int\limits_0^{2\pi} (g - \frac{\alpha}{2\pi} \cos 2t - \frac{\beta}{2\pi} \sin 2t) \cos 2t - (h + \frac{\alpha}{2\pi} \sin 2t - \frac{\beta}{2\pi} \cos 2t) \sin 2t\, dt
\\
&= \int\limits_0^{2\pi} (g \cos 2t - h \sin 2t - \frac{\alpha}{2\pi})\, dt \\
&= \int\limits_0^{2\pi} (g \cos 2t - h \sin 2t)\, dt - \alpha = 0,
\end{split}
\end{gather*}
by definition of $\alpha$, confirming the first equation in \eqref{eq_ort_relation}. The second equation is confirmed in a similar way, using the definition of $\beta$.

Thus $X_2 = X - X_1$ really does lie in $T_{\mathrm{id}}\auth^\perp$. From the explicit formulas above, one can verify that the projection maps from $\mathrm{VF}(S^3)$ onto each summand are continuous, and this completes the proof of  \autoref{prop_l2_direct_sum}.
\end{proof}

\begin{rem}
$T_{\mathrm{id}}\auth^\perp$ will play the role of the tangent space $T_{\HH}\mathrm{Fib}(S^3)$,
on which our moduli space $\mathrm{Fib}(S^3)$ will be modeled as a Fréchet manifold.
\end{rem}


\newpage

\subsection{Balanced vector fields and diffeomorphisms}

We give here an alternative interpretation of the $L^2$ orthogonal complement to the tangent space 
at the identity of the automorphism group of the Hopf fibration, in which we view the elements 
of $T_{\mathrm{id}}\auth^\perp$ as balanced vector fields on $S^3$, and their images under the exponential map as balanced diffeomorphisms of $S^3$. First, we introduce some useful constructions and definitions.

\emph{Integrating a horizontal vector field along a Hopf fiber}. If $V$ is a horizontal vector field along a Hopf fiber $F$, we project it down to a circle's worth of tangent vectors to $S^2$ at the point $p(F)$ and integrate them with respect to arc length along $F$, write the result as $\int_F V$, and regard this as a tangent vector to $S^2$ at that point.

When we divide this vector by the length $2\pi$ of the Hopf fiber $F$, we denote it by
\begin{equation}
\overline{V} = \frac{1}{2\pi} \int\limits_F V
\end{equation}
and refer to $\overline{V}$ as the \emph{average value} of $V$ along $F$.

Later we will need to lift this tangent vector $\overline{V}$ to a horizontal vector field along the fiber $F$  
in $S^3$, and will still call it $\overline{V}$ there.  We keep in mind that the Hopf projection $p$ doubles the lengths of horizontal vectors, and so lifting vectors halves their lengths. Alternatively, one can take the base space of the Hopf fibration to be a two-sphere of radius $1/2$,
so that the Hopf projection will preserve the lengths of horizontal vectors.

We note that the image $\Exp_F\, \overline{V}$ of $F$ under the diffeomorphism $\Exp\, \overline{V}$ is also a Hopf fiber.

\begin{defn}
A vector field $V$ on $S^3$ is said to be \emph{balanced} if it is everywhere horizontal and if its integral $\int_F V$ along each Hopf fiber is zero.    
\end{defn}

We will see later that the horizontal left-invariant vector fields $B$ and $C$ are balanced along each Hopf fiber.

The linear space of balanced vector fields on $S^3$ will be denoted by $\mathrm{BVF}(S^3)$.
    It is a submodule of $\mathrm{VF}(S^3)$ over the ring of smooth real-valued functions on $S^2$,  
    equivalently, over the ring of smooth functions on $S^3$ which are constant along Hopf fibers.

\begin{defn}
If $V$ is a balanced vector field on $S^3$, not too large, then its image $f = \Exp\, V$ under the Riemannian exponential map, will be called a \emph{balanced diffeomorphism} of $S^3$. The set of these will be denoted by $\mathrm{Bal}(H)$.
\end{defn}

In the next section, we will prove that the subgroup $\auth$ of automorphisms of the Hopf fibration $\HH$ is a smooth Fréchet submanifold of $\diffst$,  and the argument we give will at the same time show that a neighborhood of the identity 
in $\mathrm{Bal}(H)$ is also a smooth Fréchet submanifold of $\diffst$.

After that, we will prove that a small neighborhood of the identity in $\mathrm{Bal}(H)$ can serve as a smooth local slice for our big bundle
\begin{equation}
\auth \subset \diffst \to \mathrm{Fib}(S^3).
\end{equation}

This in turn will lead us to an alternative proof of the theorem of Rubinstein et al \cite{rubinstein} that our moduli space $\mathrm{Fib}(S^3)$ of smooth fibrations of $S^3$ by oriented simple closed curves is a smooth Fréchet manifold, and that our big bundle above is a fiber bundle in the category of smooth Fréchet manifolds and smooth maps between them.

\begin{figure}[H]
    \centering
    \includegraphics[scale=0.7]{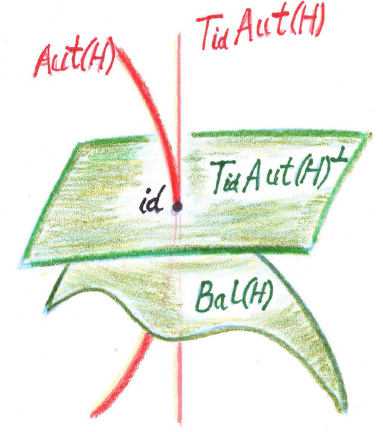}
    \caption{\centering We confirm that the tangent space at the identity to $\mathrm{Bal}(\HH)$ coincides $\text{with the $L^2$-orthogonal complement to the tangent space at the identity to $\auth$.}$}
    \label{fig:fig3}
\end{figure}

Superimposed in \autoref{fig:fig3} are portions of

\begin{itemize}
    \item[(1)] The infinite-dimensional Fréchet Lie group $\diffst$ of diffeomorphisms of the 
          three-sphere together with its Fréchet submanifolds $\auth$ of automorphisms of the Hopf fibration $\HH$, curved in red, and $\mathrm{Bal}(H)$ of balanced diffeomorphisms, curved in blue, and
    \item[(2)] The tangent space at the identity to $\diffst$, consisting of all smooth vector fields
          on the three-sphere, together with the tangent space $T_{\mathrm{id}} \auth$ at the identity to 
          $\auth$ in red, and its orthogonal complement in blue, which we will show is tangent to $\mathrm{Bal}(H)$ at the identity.
\end{itemize}

The left-invariant vector fields $A(x) = x i$, $B(x) = x j$, $C(x) = x k$ have right-invariant cousins $A^*(x) = i x$, $B^*(x) = j x$, $C^*(x) = k x$, and our left-handed Hopf fibration $\HH$ with fibers the orbits of $A$ has a right-handed cousin Hopf fibration $\HH^*$ with fibers the orbits of $A^*$.

The unit circle $F_0$ in the $(1, i)$-plane is a fiber of both $\HH$ and $\HH^*$ with the same orientation,
while the unit circle $F_1$ in the $(j, k)$-plane is a fiber of both $\HH$ and $\HH^*$ with opposite orientations.

The region between these two great circles is filled with parallel tori, with each torus filled by fibers of $\HH$, as shown in  \autoref{fig:HopfTsai}.

We visualize the projection map $p \colon S^3 \to S^2$ of the Hopf fibration $\HH$ as taking the fiber $F_0$ to the south pole, the fiber $F_1$ to the north pole, and the various tori between $F_0$ and $F_1$ to the various circles of latitude between the two poles on $S^2$.

For any one of these tori, the Hopf projection $p$ takes each fiber of $\HH$ on it to a point on the
corresponding circle of latitude on $S^2$.

But these same tori are also filled by the great circle fibers of $\HH^*$, with a sample torus shown below, and on it a fiber of $\HH$ colored red and a fiber of $\HH^*$ colored blue. Each red $(1,-1)$ fiber meets each blue $(1,1)$ fiber in a pair of points, as shown below.

\begin{figure}[H]
    \centering
    \includegraphics[scale=0.5]{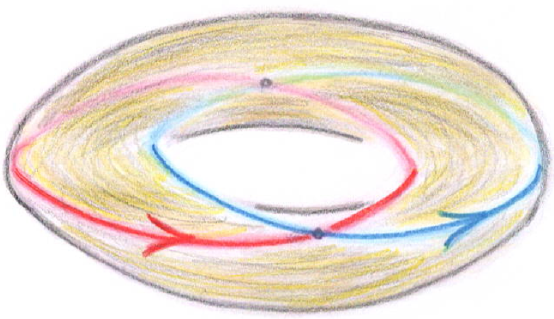}
    \caption{\centering The projection map $p$ of $\HH$ takes each fiber of $\HH^*$ to the entire circle of latitude on $S^2$, wrapping it twice uniformly around that circle.}
    \label{fig:fig4}
\end{figure}

A key observation is the following.

\begin{prop}\label{prop_exp_B}
The Riemannian exponential image $\Exp_{F_0} (\epsilon B)$ of the horizontal  
left-invariant vector field $\epsilon B$ along $F_0$ is a fiber of the right-handed Hopf fibration $\HH^*$   
lying on the torus boundary of the $\epsilon$-tubular neighborhood of the great circle $F_0$.
\end{prop}

\begin{proof}
Let the unit quaternion $q$ denote some point of $S^3$. Then $\Exp_q (\epsilon B)$ lies on the great 
circle arc between $q$ and $qj$, and hence
\begin{equation}
\Exp_q (\epsilon B) = (\cos \epsilon) q + (\sin \epsilon) qj = q (\cos \epsilon + \sin \epsilon\, j).
\end{equation}

When $q = (\cos t + \sin t\, i)$ is an arbitrary point of the Hopf fiber $F_0$, we have
\begin{equation}
\Exp_q (\epsilon B) = (\cos t + \sin t\, i)\, (\cos \epsilon + \sin \epsilon\, j).
\end{equation}

Because the factor $(\cos t + \sin t\, i)$ above is on the left, then as $t$ varies between $0$ and $2\pi$,  
these points above trace out a fiber of the Hopf fibration $\HH^*$ through the point $\cos \epsilon + \sin \epsilon\, j$,   
lying on the torus boundary of the $\epsilon$-tubular neighborhood  of $F_0$,  as claimed.
\end{proof}

\begin{cor}
The horizontal left-invariant vector fields $B$ and $C$ are balanced along the 
fibers of the Hopf fibration $\HH$.
\end{cor}

\begin{proof}
Since $\Exp_{F_0} (\epsilon B)$ is a fiber of $\HH^*$, it is taken by the projection map $p$ of $\HH$ twice 
around a circle of latitude on $S^2$. Therefore the horizontal vector field $\epsilon B$ along $F_0$ is taken 
uniformly twice around a circle of radius $2\epsilon$ in the tangent plane to $S^2$ at the south pole.   
Hence its integral there is zero, and so $\epsilon B$ is balanced along $F_0$.

Likewise, $B$ and $C$ are balanced along $F_0$, and then by left-invariance, they are balanced along 
every fiber of $\HH$.
\end{proof}

\newpage

\begin{theorem}
The $L^2$ orthogonal complement $T_{\mathrm{id}}\auth^\perp$ to the tangent space at the identity of the automorphism group $\auth$ coincides with the subspace of horizontal vector fields on $S^3$ which are balanced with respect to the Hopf fibration $\HH$. In other words, $\mathrm{PVF}(S^3) \oplus \mathrm{BVF}(S^3) = \mathrm{VF}(S^3)$.
\end{theorem}

\begin{proof}
By \autoref{taut_perp_description}, a horizontal vector field $Y = g\,B + h\,C$ is orthogonal to $T_{\mathrm{id}}\auth$
if and only if along each Hopf fiber $F$ we have
\begin{equation}\label{local_ort_condition}
\int\limits_0^{2\pi} (g \cos 2t - h \sin 2t) \, dt = 0 \qquad \text{and} \qquad \int\limits_0^{2\pi} (g \sin 2t + h \cos 2t) \, dt = 0 .
\end{equation}
From \autoref{prop_exp_B}, we can choose coordinates on the tangent plane $T_{p(F)}S^2$
as the ordinary $xy$-plane, and mindful that $B$ and $C$ are orthogonal to one another, write:
\begin{equation}
p_*B = 2(\cos 2t , \sin 2t) \qquad \text{and} \qquad p_*C = 2(-\sin 2t , \cos 2t) .
\end{equation}
Now the condition that the horizontal vector field $Y = g\,B + h\,C$ be balanced is that
for each Hopf fiber we have $\int\limits_0^{2\pi} p_*Y\, dt = 0$. But then
\begin{equation}
\begin{split}
p_*Y &= g\, p_*B + h\, p_*C \\
     &= g \cdot 2 (\cos 2t, \sin 2t) + h \cdot 2 (-\sin 2t, \cos 2t) \\
     &= 2 \big( g \cos 2t - h \sin 2t, \; g \sin 2t + h \cos 2t \big) .
\end{split}
\end{equation}
Hence $\int\limits_0^{2\pi} p_*Y\, dt = 0$ if and only if
\begin{equation}
\int\limits_0^{2\pi} (g \cos 2t - h \sin 2t)\, dt = 0 \qquad \text{and} \qquad \int\limits_0^{2\pi} (g \sin 2t + h \cos 2t)\, dt = 0 .
\end{equation}
But this is precisely the condition in equation \eqref{local_ort_condition} that $Y =$ $g\,B$ $+$ $h\,C$ be orthogonal to $T_{\mathrm{id}}\auth$.
This shows that the set of balanced horizontal vector fields on $S^3$ coincides with
the orthogonal complement $T_{\mathrm{id}}\auth^\perp$ of $T_{\mathrm{id}}\auth$ inside $\mathrm{VF}(S^3)$, as claimed.
\end{proof}

\newpage

\section{$\aut$ is a smooth Fr\'echet submanifold of $\diff$}\label{submanifoldsection} 

In this section we will introduce the \emph{aligned Riemannian exponential map}, a variant of the
ordinary Riemannian exponential map which is nicely adapted to the Hopf fibration of $S^3$, and show that both of these exponential maps induce the same smooth Fr\'echet manifold
structure on the diffeomorphism group $\diffst$.

Then we will use this to prove Theorem \ref{submanifold_thm}, that $\auth$ is a smooth Fr\'echet submanifold of $\diffst$. First a picture to guide us, and a detailed description of it.

\begin{figure}[H]
    \centering
    \includegraphics[scale=0.6]{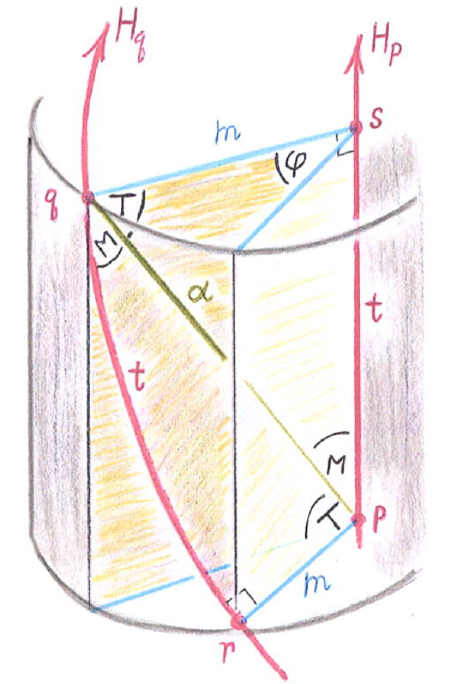}
    \caption{\centering Open book decomposition of $S^3$ with Hopf fiber $\HH_p$ as binding.}
    \label{fig:fig5}
\end{figure}
\noindent\textbf{Description of \autoref{fig:fig5} and setup for the proof of \autoref{submanifold_thm}}.

\begin{itemize}
    \item Let $V$ be any smooth vector field on $S^3$, close to the zero field.
\item Let $p$ be any point on $S^3$, at which we start.
\item Let $q = \big(\Exp V\big)(p)$, where $\Exp$ is the Riemannian exponential map. Smooth dependence of geodesics on initial conditions tells us that $q$ depends smoothly on $p$ and $V$.
\item Consider the oriented Hopf fibers $\HH_p$ through $p$ and $\HH_q$ through $q$, shown in red.
\item Shown in blue are two minimizing geodesic arcs of length $m$ between these fibers, one from $p$ to $r$ at the bottom of the figure, and the other from $q$ to $s$ at the top.
Smoothness of the nearest neighbor map tells us that $r$, $s$ and $m$ depend smoothly on $p$ and $V$.
\item Shown in green is the geodesic arc $\Exp\, V$ from $p$ to $q$, of length $\|\Exp\,V\| = \| V \| = \alpha$.
\item This green geodesic arc, labeled $\alpha$, can help us to visualize its tangent vector $V(p)$.
\item The \emph{horizontal run} between $p$ and $q$ is the distance $m$.
\item The \emph{vertical rise} between $p$ and $q$ is the signed distance $t$ along $\HH_p$ from $p$ to $s$,
which is equal to the signed distance along $\HH_q$ from $r$ to  $q$. This vertical rise also depends smoothly on $p$ and $V$.
\item Consider the open book decomposition of $S^3$ with binding $\HH_p$ and with pages the circle's worth of great hemispheres with $\HH_p$ as equator. The two ``rectangles'' in Figure \ref{fig:fig5} are contained in two pages of this open book decomposition. The right page is spanned by $\HH_p$ and the lower blue geodesic from $p$ to $r$, while the left page is spanned by $\HH_p$ and the upper blue geodesic from $q$ to $s$. The \emph{turn angle} is the angle $\varphi$ between these two pages.
\item Note that the green geodesic arc from $p$ to $q$ of length $\alpha$ must lie on the left page of the book, since its endpoints do and since these pages are totally geodesic.
Hence the tangent vector $V(p)$ is tangent to this left page.
\item The Hopf fiber $\HH_q$ cuts steadily across the pages of our open book decomposition. A full trip around $\HH_q$ has length $2\pi$, during which it cuts across $2\pi$ worth of pages, the full book. It follows that the portion of $\HH_q$ of length $t$ shown above cuts across just those many pages of our book, and hence that the vertical rise $t$ between $p$ and $q$ equals the turn angle $\varphi$ shown between the two pages of the book: $t = \varphi$.

\end{itemize}

Since our goal is to prove that $\auth$ is a smooth Fr\'echet submanifold of $\diffst$, we recall the precise definition.

\begin{defn}[\cite{Ham}]
Let $M$ be a Fr\'echet manifold and $N$ a closed subset. Then $N$ is a \emph{Fr\'echet submanifold} of $M$ if every point of $N$ lies in the domain of a coordinate chart on $M$
with range in a product of Fr\'echet spaces $F \times G$ such that a point in the domain of the chart lies in the subset $N$ if and only if its image under the chart lies in the subset $F \times \{0\}$.
\end{defn}

Now let $V$ be any smooth vector field on $S^3$ close to the zero field. Convert $V$ to another smooth vector field $X$ on $S^3$ in two steps as follows:

First, \emph{rescale} the vertical and horizontal components of $V$ separately to convert $V$ into the vector field $W$,
\begin{equation}\label{V_to_W}
W = t\dfrac{V^{\|}}{\| V^{\|} \|} + m\dfrac{V^{\perp}}{\| V^{\perp} \|}   
\end{equation}
where at each point $p \in S^3$, the tangent vectors $V^{\|}$ and $V^{\perp}$ are the components of $V$ parallel
to and orthogonal to the Hopf fiber through $p$, and the coefficients $t$ and $m$ are the vertical rise and horizontal run of $\Exp V(p)$ shown in  \autoref{fig:fig5} and explained in its description.

\begin{rem}
Formula \eqref{V_to_W} for $W$ does not make sense when either $V^{\|}$ or $V^{\perp}$ is zero, so we will give an alternative expression for it to confirm that it is well-defined and depends smoothly on $V$.
\end{rem}

Second, \emph{twist} $W$ about each point $p$ of each Hopf fiber through an angle equal to the vertical rise $t$ of $\Exp\, V(p)$ to get the vector field $X$ on $S^3$.

We will define the aligned Riemannian exponential map $\aexp$ in a moment, but state in advance

\begin{prop}\label{prop_aexp_exp}
$\aexp(X) = \Exp(V)$, where the vector fields $V$ and $X$ are given above, or alternatively expressed:
\begin{equation}
\aexp \circ \mathrm{Twist} \circ \mathrm{Rescale} = \Exp.
\end{equation}
Furthermore, both \emph{Rescale} and \emph{Twist} are diffeomorphisms between open neighborhoods of the zero vector field in $\vf$, and hence $\aexp$ and $\Exp$ define the same smooth Fr\'echet manifold structure on $\diffst$.
\end{prop}

\medskip

\begin{prop}\label{prop_aexp_neighborhood}
The aligned Riemannian exponential map $\aexp$ takes a neighborhood of the zero vector field in $\mathrm{VF}(S^3)= T_{\mathrm{id}}\auth\oplus   T_{\mathrm{id}}\auth^\perp$
to a $\diffst$ neighborhood of any diffeomorphism
in $\auth$ so that the vectors in $T_{\mathrm{id}}\auth \oplus 0$ are matched with diffeomorphisms in $\auth$.
\end{prop}
This will give us \autoref{submanifold_thm}.

\subsection{Definition of the aligned Riemannian exponential map}

Recall that the Riemannian exponential map $\Exp$ takes a vector field $V$ on $S^3$ to a
diffeomorphism $\Exp(V)$ of $S^3$, which sends any point $p$ on $S^3$ to the endpoint of the
geodesic arc beginning at $p$ and tangent there to the vector $V(p)$, and of length equal to the length of this vector.

Thanks to the existence and uniqueness of geodesics with prescribed initial conditions, and to smooth dependence on parameters, a vector field $V$ in a small $C^{\infty}$ neighborhood of the zero field exponentiates to a diffeomorphism of $S^3$ which is $C^{\infty}$ close to the identity.

The aligned Riemannian exponential map takes a vector field $X$ on $S^3$ to a diffeomorphism $\aexp(X)$ of $S^3$ which is defined with respect to a given Hopf fibration of $S^3$ as follows.

First we decompose the vector field $X$ into components parallel to and orthogonal to the Hopf fiber $\HH_p$ through each point $p \colon X = X^{\|} + X^{\perp}$.

Then $p$ is moved a distance $t = \| X^{\|} \|$ along this Hopf fiber to the temporary location ${s = \Exp X^{\|}(p)}$, shown in \autoref{fig:fig5}.

At the same time, the vector $X^{\perp}(p)$ is moved along this Hopf fiber up to the point $s = \Exp X^{\|}(p)$, turning about the fiber so that its projection to the base $S^2$ of the Hopf fibration never changes.
This rotated version of $X^{\perp}(p)$, now located at the point $s$, then directs a geodesic arc orthogonal to the Hopf fibers and of length $m = \|X^{\perp}(p)\|$, ending at the point $q$.

We define $\aexp X(p) = q$.

For the same reasons as above, vector fields close to zero exponentiate this way to diffeomorphisms of $S^3$ close to the identity. In  \autoref{fig:fig5} above, we see that $\aexp X(p) = q = \Exp V(p)$.

\subsection{Riemannian exponentials and logarithms on $S^3$}

We repeat from the previous page that if $p$ is a point on $S^3$ and $V$ is a vector tangent to $S^3$ at $p$,
then the Riemannian exponential map at $p$ takes the vector $V$ to the endpoint $q$ of the geodesic
arc on $S^3$ which starts at $p$ tangent to $V$ and has the same length as $V$. We write $q = \Exp_p V$.

If $p$ and $q$ are points on $S^3$ which are not antipodal, then the \emph{Riemannian logarithm} at $p$ takes $q$ to the tangent vector $V$ at $p$ for which $q = \Exp_p V$. We write $\Log_p(q) = V$.

Next we allow $p$ and $V$ to vary as much as possible.
Let $TS^3$ denote the tangent bundle of $S^3$, namely
\begin{equation}
TS^3 = \{ (p , V) \,:\, p \in S^3 ~\text{and}~ V \in T_pS^3 \}.
\end{equation}
Then the Riemannian exponential map $\Exp$ will take an open subset of $TS^3$ to an open subset of $S^3 \times S^3$ as follows:
\begin{equation}
\Exp \colon \{ (p , V) : \|V\| < \pi \} \rightarrow \{ (p, q) : q \neq -p \},
\end{equation}
defined by $\Exp(p, V) = \big(p, \, \Exp_p(V)\big)$.

Smooth dependence of geodesics on initial conditions tells us that this version of the Riemannian exponential map is a diffeomorphism from a fat tubular neighborhood of the zero section in $TS^3$ to almost all of $S^3 \times S^3$, and its inverse is the Riemannian logarithm.
\begin{prop}
The rescaled vector field
\begin{equation}
W = t \dfrac{V^{\|}}{\| V^{\|} \|} + m \dfrac{V^{\perp}}{\| V^{\perp} \|}
\end{equation}
depends smoothly on the vector field $V$ in a neighborhood of the zero field in $\vf$. 
\end{prop}

\begin{proof}
We need to rewrite these two summands of $W$ so that they make sense even when $V^{\|}$ or $V^{\perp}$ is zero, and do this as follows. For the first summand, we have 
\begin{equation}
t \dfrac{V^{\|}}{\| V^{\|} \|} = \Log_p(s),
\end{equation}
where $s$ is shown in  \autoref{fig:fig5} and defined right afterwards, where it is noted that $s$ depends smoothly on $p$ and $V$. For the second summand, we have
\begin{equation}
m \dfrac{V^{\perp}}{\| V^{\perp} \|} = \mathcal{P}\big(\Log_s(q)\big),
\end{equation}
where $\mathcal{P}$ indicates parallel transport of the tangent vector $\Log_s(q)$ to $S^3$ at $s$ down along the geodesic arc of the Hopf fiber $\HH_p$ from $s$ back to $p$.
So we can write
\begin{equation}
W = \Log_p(s) + \mathcal{P}\big(\Log_s(q)\big),
\end{equation}
which lets us see that $W$ depends smoothly on the vector field $V$ in a neighborhood of the zero field in $\vf$.
\end{proof}

\begin{prop}
The rescaling map $V \rightarrow W$ is a diffeomorphism between neighborhoods of the zero vector field in $\vf$.
\end{prop}

\begin{proof}
To show this, we will smoothly invert the map $V \rightarrow W$ by carrying out in order the following list of
actions, each depending in a smooth way on those listed earlier.

We begin with $W = \Log_p(s) + \mathcal{P}\big(\Log_s(q)\big)$ and will smoothly split it into its summands
\begin{equation}
W^{\|} = \Log_p(s) \qquad \text{and} \qquad W^{\perp} = \mathcal{P}\big(\Log_s(q)\big)
\end{equation}
which are parallel to and orthogonal to the Hopf fibers. To do this, we note that
\begin{equation}
W^{\|} = \Log_p(s) = \langle W, A \rangle A
\end{equation}
where $A(p) = pi$ (quaternion multiplication) is the left-invariant unit vector field on $S^3$ tangent to the Hopf fibers. This extracts $W^{\|}$  smoothly from $W$, and then $W^{\perp} = W - W^{\|}$.

From $W^{\|}$ we get
\begin{equation}
\Exp_p\big(W^{\|}(p)\big) = \Exp_p \circ \Log_p(s) = s.
\end{equation}

Knowing $s$, we parallel transport $W^{\perp} = \mathcal{P}\big(\Log_s(q)\big)$ up along the geodesic arc of the Hopf fiber $\HH_p$ from $p$ to $s$ and recover $\Log_s(q)$.

From $\Log_s(q)$ we get $\Exp_s \circ \Log_s(q) = q$ to recover $q$.

From $q$ we get $\Log_p(q) = V(p)$ to get $V$. This shows how to invert the rescaling map $V \rightarrow W$ to get $W \rightarrow V$, and we conclude that these maps are diffeomorphisms between neighborhoods of the zero field in $\vf$.
\end{proof}

We noted earlier that the vertical rise $t = t(V, p)$ depends smoothly on the vector field ${V \in \vf}$ and the point $p \in S^3$.

Next, write $W$ as a linear combination of the orthonormal basis of left-invariant vector fields $A(p) = pi$, $B(p) = pj$ and $C(p) = pk$:
\begin{equation}
W(p) = f(p)A(p) + g(p)B(p) + h(p)C(p)
\end{equation}
with $C^{\infty}$ coefficients $f$, $g$ and $h$. We will express the twist of $W$ through the angle $\varphi = t(V, p)$ in terms of this decomposition as follows.

 \autoref{fig:fig5} shows the Hopf fiber $\HH_q$ coiling left-handedly around the Hopf fiber $\HH_p$, so we will twist the vector $W(p)$ right-handedly through an angle $\varphi = t(V, p)$ to bring it from the left page to the right page.
Since $A(p)$, $B(p)$, $C(p)$ is a right-handed frame at $p$, we have
\begin{equation}
X(p) = f(p)A(p) + \Big(\!\cos(t) g(p) - \sin(t)h(p)\Big) B(p) + \Big(\sin(t) g(p) + \cos(t) h(p)\Big) C(p),
\end{equation}
where we have shortened $t(V, p)$ to $t$ above for the sake of clarity. This shows that $X$ is a smooth vector field on $S^3$ and that the operation $\mathrm{Twist} \circ \mathrm{Rescale}$,
taking $V \rightarrow W \rightarrow X$, is a smooth map between neighborhoods of the zero vector field in $\vf$.

To see why $\mathrm{Twist} : W \rightarrow X$ is a diffeomorphism between neighborhoods of the zero field in $\vf$, note that 
the vertical rise $t(p)$ is the length of the vertical component of $W$, and hence also the length of the vertical component of $X$, and so depends smoothly on $X$.
Hence untwisting the vector $X(p)$ through the angle $\varphi = t(p)$ about the Hopf fiber $\HH_p$ through $p$ smoothly returns it to the vector field $W$.

Thus $\mathrm{Twist} \colon W \rightarrow X$ is a diffeomorphism between neighborhoods of the zero field in $\vf$.

Since both $\mathrm{Rescale}$ and $\mathrm{Twist}$ are diffeomorphisms between neighborhoods of the zero field in $\vf$, the same is true of their composition
$\mathrm{Twist} \circ \mathrm{Rescale}$. Therefore both
$$\aexp \qquad \text{and} \qquad \Exp = \aexp \circ \mathrm{Twist} \circ \mathrm{Rescale}$$
define the same smooth Fr\'echet manifold structure on $\diffst$. This proves \autoref{prop_aexp_exp}.

Suppose now that the vector field $X$ on $S^3$ is ``aligned'' with the projection of the Hopf bundle in the sense that it projects to a well-defined vector field on $S^2$, equivalently, that $X \in \taut$. Then the diffeomorphism $\aexp(X)$ is an automorphism of the Hopf fibration, and not otherwise.

So $\aexp$ takes a neighborhood of the zero vector field in 
\begin{equation}
\vf = \taut \oplus \taut^{\perp}
\end{equation}
to a $\diffst$ neighborhood of the identity in $\auth$, with the vectors in $\taut \oplus 0$ going to automorphisms in $\auth$. This proves \autoref{prop_aexp_neighborhood}.

Composition with automorphisms of $\HH$ then moves these nice coordinates to neighborhoods of all the elements of $\auth$.
This shows, according to Hamilton's definition of Fr\'echet submanifold, stated earlier, that $\auth$ is a smooth Fr\'echet submanifold of $\diffst$.

This proves  \autoref{submanifold_thm}.

\newpage

\section{A smooth slice for the sequence $\Aut(H) \subset \Diff(S^3)\to \Fib(S^3)$}
\label{slice_section}
Our goal here is to produce the most natural smooth slice possible, with image on the local Fréchet submanifold of $\diffst$ consisting of balanced diffeomorphisms, because this is the Riemannian exponential image of the $L^2$-orthogonal complement to the tangent space at the identity of the Fréchet subgroup of automorphisms of the Hopf fibration.

\subsection{Hunting for a smooth slice of balanced diffeomorphisms}

In  \autoref{fig:fig7} below, we start with a coset $f \, \auth$ which is close to $\auth$, and for convenience first modify the representative $f$ so that it is close to the identity and moves points horizontally. Then we will search for and find a unique balanced diffeomorphism in this same coset.

\begin{figure}[H]
    \centering
    \includegraphics[scale=0.7]{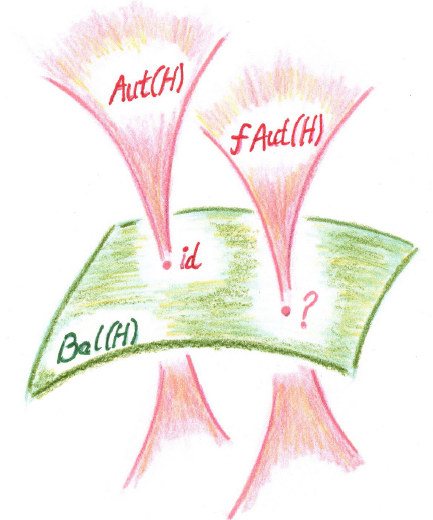}
    \caption{\centering We must show that the left coset $f\,\aut$ \\ contains a unique balanced diffeomorphism of $S^3$.}
     \label{fig:fig7}
\end{figure}

This will prove the following result.

\begin{theorem}[Smooth slice theorem for $S^3$] 
The exact sequence
\begin{equation}
\auth \subset \diffst \to \mathrm{Fib}(S^3)
\end{equation}
admits a smooth slice consisting of balanced diffeomorphisms, and is therefore a smooth fiber bundle in the Fréchet category.
\end{theorem}

\newpage

In a similar manner to the previous section, we rely on \autoref{fig:fig8} to guide us, and follow it by a list of definitions used in this section.
\begin{figure}[H]
    \centering
    \includegraphics[scale=0.5]{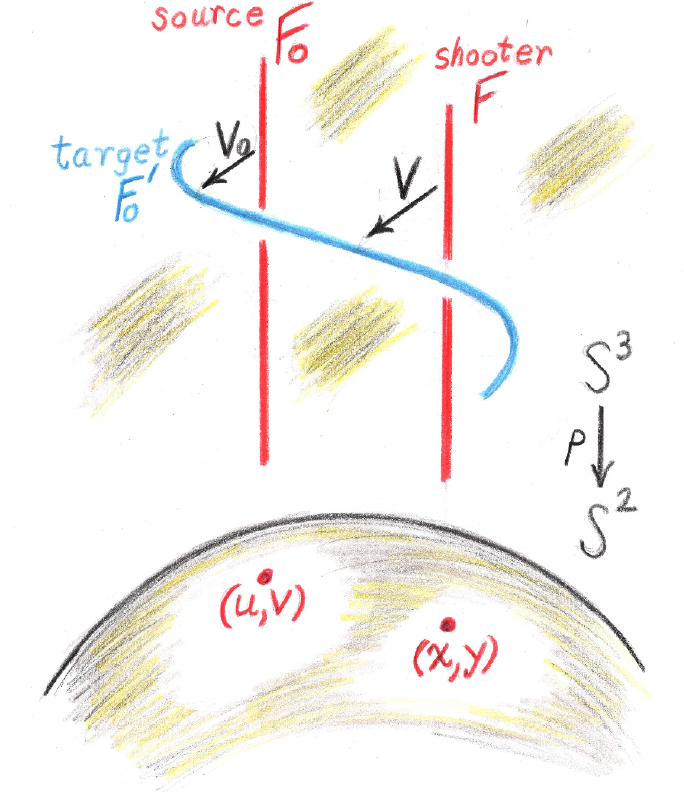}
    \caption{\centering Nearby fibers $F$ shoot horizontal vector fields $V$ to hit the image \\ blue fiber $F_0'$. The winning fiber will be the one whose vector field $V$ is balanced.}
    \label{fig:fig8}
\end{figure}

\noindent\textbf{Description of \autoref{fig:fig8}}. Let $f \auth$ be a given left coset of $\auth$, near $\auth$ itself. We use the following notation in the picture:
\begin{itemize}
     \item $F_0$ = a chosen ``source'' Hopf fiber, kept fixed until near the end;
     \smallskip
     \item $f(F_0) = F_0' = $ blue image fiber;
     \smallskip
     \item $F$ = nearby Hopf fiber, shooting for $F_0'$;
     \smallskip
     \item $V = V(t)$ = horizontal vector field from $F$ to $F_0'$;
     \smallskip
     \item $(u, v)$ = location on $S^2$ of the projection $p(F_0)$ of the source Hopf fiber.
     \smallskip
     \item $(x, y)$ = location on $S^2$ of the projection $p(F)$ of the shooter Hopf fiber;
\end{itemize}

Central roles in the argument to follow will be played by the \emph{average value}
\begin{equation}
\overline{V} = \frac{1}{2\pi} \int\limits_F V(t) dt
\end{equation} 
of the vector field $V$ along the Hopf fiber $F$, and by the \emph{energy}
\begin{equation}
E(V \text{ on } F) = \frac{1}{2\pi} \int\limits_F \frac{1}{2} \|V(t)\|^2 dt
\end{equation} 
of this vector field.

The energy $E$ is a function of the four variables $(u, v, x, y)$, or equivalently a function of the location of the shooter Hopf fiber $(x, y)$ depending smoothly on parameters $(u, v)$ indicating the location of the source fiber.

To begin, we focus on a single Hopf fiber $F_0$ and keep it fixed for a while. Its image $f(F_0) = F_0'$ is the blue image fiber that is the target of our hunt. Each nearby Hopf fiber $F$ shoots for this target with a horizontal vector field $V$, meaning that $\Exp_F V = F_0'$, and wins the hunt if the average value $\overline{V}$ of the shooting vector field $V$ along the fiber $F$ is zero. We need to ensure not just existence and uniqueness of the winning Hopf fiber, but also smooth dependence on initial conditions. More precisely, we need that the location $(x, y)$ of the winning shooter fiber $F$ depends smoothly on the location $(u, v)$ of the source fiber $F_0$ (see  \autoref{fig:fig8}). 

\begin{theorem}\label{unique_min_prop}
There exists a unique minimum point $m$ of the energy function $E$ on an open subset of $S^2$, which varies smoothly with the choice of initial Hopf fiber $F_0$.
\end{theorem}

The proof of our Smooth Slice Theorem for $S^3$ will follow quickly from  \autoref{unique_min_prop}. Note that even when a family of smooth functions depending smoothly on parameters each has a unique minimum, the location of these minima may fail to depend smoothly on parameters. A simple example of this is the family of functions $f_t(x) = x^4/4 - t^2x$, with $t$ as parameter, because the unique minimum occurs at $x = t^{2/3}$. In our attempt to prove the smooth slice theorem, this is precisely the kind of problem we encounter. We will deal with it by introducing some convexity into the problem. 

The key ingredient in the proof of \autoref{unique_min_prop} will be  \autoref{hessian_prop} which establishes strong convexity of the energy function $E$. Existence and uniqueness of the winning Hopf fiber follow from this immediately, and in combination with the Implicit Function Theorem, we will also deduce smooth dependence on the initial conditions.

Recall that \emph{strongly convex} means that a function $f$ is convex thanks to some strict derivative inequality, which for us will be that its Hessian is positive definite. \emph{Strictly convex} means that the interior of the line segment (geodesic) connecting two points on or above the graph of $f$ lies strictly above the graph. Strongly convex implies strictly convex, but not vice versa: for example, the function $f(x) = x^4$ is strictly convex, but not strongly convex.

\begin{proposition}\label{hessian_prop}
The Hessian of our energy function $E$ is positive definite, and hence the energy function $E$ is strongly convex for each choice of initial Hopf fiber $F_0$.
\end{proposition}

In order to prove this proposition, we first prove two intermediate technical results, \autoref{gradient_prop} and \autoref{covariant_prop}.

\newpage

\begin{lem}\label{gradient_prop}
The gradient of the energy function is the negative of the balance field $\overline{V}$ on the base space $S^2$:
\begin{equation}
\nabla E = -\overline{V}.
\end{equation}
\end{lem}

\begin{proof}
\begin{figure}[H]
    \centering
    \includegraphics[scale=0.4]{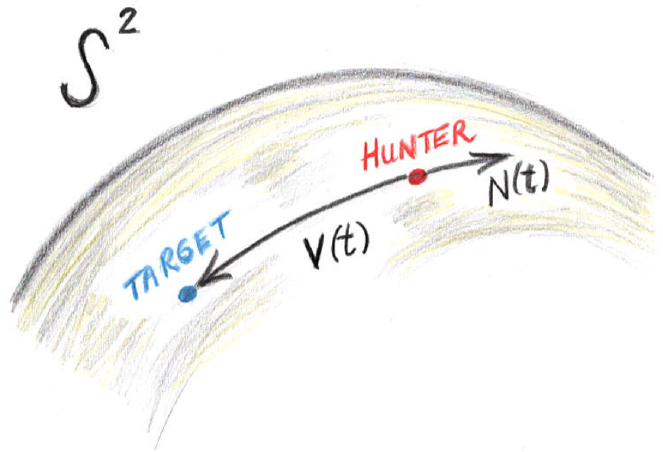}
    \caption{\centering A point at location $t$ on the red hunter fiber $F$ in $S^3$ \\ shoots a horizontal vector $V(t)$ to a point on the blue image image fiber $F_0'$, \\ and all this is projected down to the base $S^2$.}
    \label{fig:fig9}
 \end{figure}
The unit vector $N(t)$ shown in \autoref{fig:fig9}  is the gradient of the distance function between these points. Then
\begin{equation}
\frac{1}{2} \nabla \|V(t)\|^2 = \|V(t)\| N(t) = -V(t).
\end{equation}
Next, integrating along the red hunter fiber $F$, we compute the gradient of the energy $E$ of the vector field $V(t)$ along $F$ to be

\begin{equation}
\nabla E = \frac{1}{2\pi} \nabla \int\limits_F \frac{1}{2} \|V(t)\|^2 dt = \frac{1}{2\pi} \int\limits_F \frac{1}{2} \nabla \|V(t)\|^2 dt
= \frac{1}{2\pi} \int\limits_F -V(t) dt = -\overline{V}.
\end{equation}
This appears in Karcher (1977) as formula (1.2.1) for a slightly different energy function.
\end{proof}

\begin{lem}\label{covariant_prop}
The covariant derivative $\nabla_U \overline{V}$ of the balance field $\overline{V}$ can be made arbitrarily close to the vector $-U$ in the $C^0$ topology for sufficiently small open sets on $S^2$.
\end{lem}

We note that in the flat setting of the product bundle $S^1 \subset S^1 \times S^1 \times S^1 \rightarrow S^1 \times S^1$, the covariant derivative $\nabla_U \overline{V} = -U$ exactly. The deviation from this in the case of the Hopf bundle ${S^1 \subset S^3 \rightarrow S^2}$ is due to the curvature of the base and of the bundle itself. This bundle curvature appears as holonomy in the corresponding bundle of tangent $2$-planes orthogonal to the Hopf fibers, i.e., the standard tight contact structure on $S^3$. 

\begin{proof}
\begin{figure}[H]
    \centering
    \includegraphics[scale=0.6]{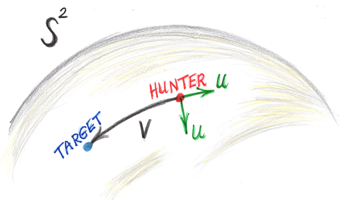}
    \caption{\centering How will the shooting vector $V$ change \\ when the hunter is moved a little while the target stays the same?}
    \label{fig:fig10}
\end{figure}
To begin, we ignore the Hopf bundle and focus on its base space $S^2$, and in \autoref{fig:fig10} show there a single hunter point in red, a single target point in blue, and a shooting vector $V$ from hunter to target, meaning that $\Exp_{\mathrm{hunter}} V = \mathrm{target}$. The issue is to quantify how the shooting vector $V$ will change when the hunter point moves a little bit away while the target point stays fixed.

The covariant derivative $\nabla_U V$ at the point $x$ depends linearly on the tangent vector $U$, so in the figure, it will suffice to learn the value of $\nabla_U V$ at $x$ when $U$ is parallel to $V$ at $x$ and when it is orthogonal to $V$ there.

When $U$ is parallel to $V$ at $x$, it is easy to see that $\nabla_U V = -U$.

When $U$ is orthogonal to $V$ at $x$, the curvature of $S^2$ will come into play, and we will see that the value of $\nabla_U V$ at $x$ is numerically a little larger than in the flat Euclidean plane $\mathbb{R}^2$:
\begin{equation}\label{eq_local_cov_uv}
\nabla_U V(x) = \left(\frac{\alpha}{\sin \alpha}\right)(-U(x))
\end{equation}
where $\alpha = |V(x)|$.
To prove \eqref{eq_local_cov_uv}, we key the discussion to \autoref{fig:fig11}, representing a neighborhood of the north pole $(0,0,1)$ of the unit 2-sphere in $\mathbb{R}^3$, with axes chosen as shown. 
\begin{figure}[H]
    \centering
    \includegraphics[scale=0.4]{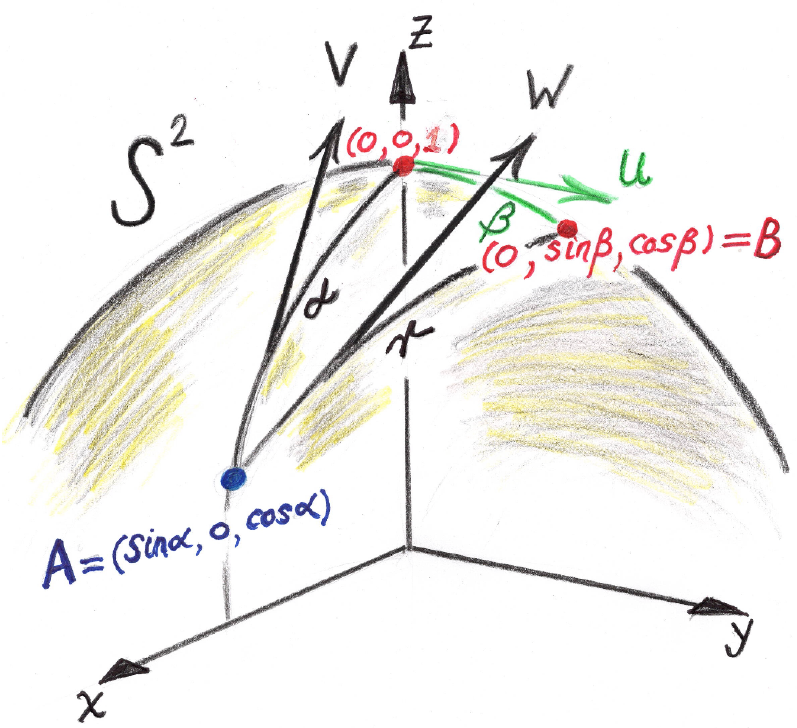}
    \caption{\centering The covariant derivative $\nabla_UV(x)$.}
    \label{fig:fig11}
\end{figure}
Note that we have, just for this discussion, reversed the arrows on the vectors $V$ and $W$ so that they point from target to hunter.

A great circle arc in the $xz$-plane runs from the target point $A = (\sin \alpha, 0, \cos \alpha)$ to the hunter point at the north pole $(0,0,1)$. This arc has length $\alpha$ and is the exponential image of the tangent vector $V$ as shown, so that $V = \alpha(-\cos\alpha, 0, \sin \alpha)$.

From the north pole, we move a bit to the right to the point $B = (0, \sin\beta, \cos\beta)$ in the $yz$-plane. The small arc of great circle between these two points has length $\beta$ and is the exponential image of the tangent vector $U$ as shown.

A great circle arc runs from fixed target point $A$ to perturbed hunter point $B$. It has length $\gamma$ and is the exponential image of the tangent vector $W$ as shown.

The spherical Pythagorean formula tells us that $(\cos\alpha)(\cos\beta) = \cos \gamma$. Direct computation gives the explicit formula
\begin{equation}
W = \left(\frac{\gamma}{\sin \gamma}\right)\big(\!-\!\sin\alpha\, \cos\alpha\, \cos\beta,\; \sin\beta,\; \sin 2\alpha\, \cos\beta\big).
\end{equation}

A quick check, using the spherical Pythagorean formula, confirms that $\|W\| = \gamma$. Another quick check, putting $\beta = 0$, confirms that $W|_{\beta=0} = V$. A direct computation shows that
\begin{equation}
\left.\frac{dW}{d\beta}\right|_{\beta=0} = \left(\frac{\alpha}{\sin\alpha}\right)(0, 1, 0).
\end{equation}
This leads quickly to the formula \eqref{eq_local_cov_uv}.  If we agree in advance to keep all the action within a hemisphere on $S^2$, then we have $\alpha = \|V(x)\| \leq \pi/2$, and hence
\begin{equation}
1 \leq \frac{\alpha}{\sin \alpha} \leq \frac{\pi}{2} \leq 1.6.
\end{equation}
Then, thanks to linearity of $\nabla_U V$ at $x$, we will have
\begin{equation}\label{eq_cov_uv_2}
\nabla_U V(x) = \lambda(-U(x)), \quad \text{with} \quad 1 \leq \lambda \leq 1.6.
\end{equation}
In particular, the covariant derivative $\nabla_U V$ can be made arbitrarily close in the $C^0$ topology to the vector $-U$ for fixed target and for hunters in a sufficiently small open subset of $S^2$.

Regarding formula \eqref{eq_local_cov_uv} above, we note that the boost $\alpha/\sin \alpha$ in apparent sidewise speed conforms to common sense, since a circle of radius $\alpha$ in the plane has circumference $2\pi \alpha$, while on the unit sphere it only has circumference $2\pi \sin\alpha$. Thus, with eyes at the center of the circle, a particle moving around the circle at given speed appears to have greater angular velocity on the unit sphere than in the plane by the factor of $\alpha/\sin\alpha$.

We pause now to discuss holonomy in the Hopf bundle, and after that will return to complete the proof of \autoref{covariant_prop}.

\newpage

\noindent \emph{Holonomy in the Hopf bundle.} The distribution of tangent 2-planes in $S^3$ orthogonal to the Hopf fibers is the standard tight contact structure on $S^3$. It is non-integrable, as depicted in \autoref{fig:fig12}, and our job will be to see how this affects the story at hand.
\begin{figure}[H]
    \centering
    \includegraphics[scale=0.5]{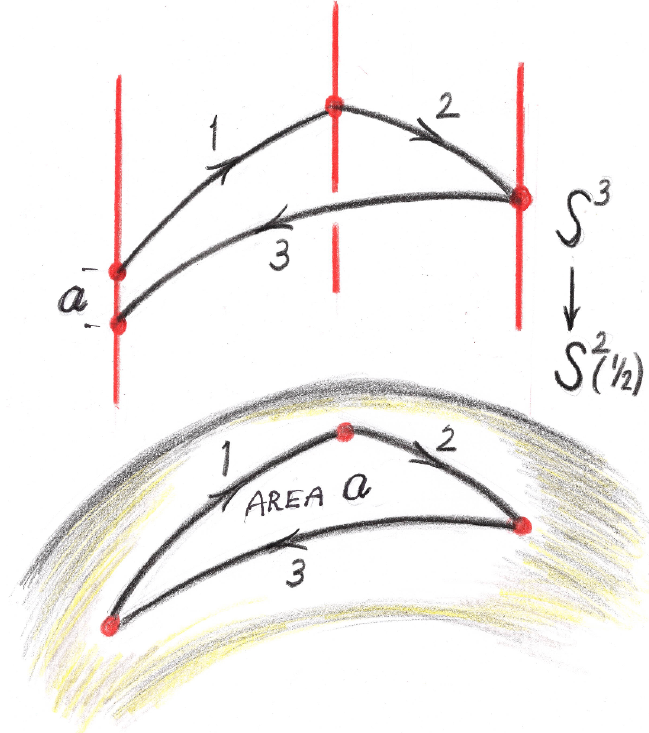}
    \caption{\centering Holonomy in the Hopf bundle.}
    \label{fig:fig12}
\end{figure}
In \autoref{fig:fig12}, we show three Hopf fibers in $S^3$, and horizontal geodesics between them, drawn in the labeled order. When we come back to the first geodesic, we are off a bit from the starting point. But when we project Hopf fibers and horizontal geodesics down to $S^2$, the triangle closes up.

If for convenience we take the base $S^2$ to have radius $1/2$, so that the projection map $p$ of the Hopf fibration is a Riemannian submersion, then the area of the triangle on $S^2$ equals the amount of separation between starting and ending points on the left-most Hopf fiber in $S^3$.

This is easily confirmed by using the left invariant one-forms $A^{\flat}, B^{\flat}, C^{\flat}$ dual to the left-invariant vector fields $A, B, C$ on $S^3$, where the Hopf fibers are the integral curves of $A$, and where \newline ${dA^{\flat} = -2B^{\flat} \wedge C^{\flat}}$, which gives the area 2-form on $S^2(1/2)$. 

\newpage

\noindent\emph{Back to our proof.} To finish the proof of \autoref{covariant_prop}, we turn our attention to what happens when Hopf fibers hunt for a fixed blue target image fiber. 

\begin{figure}[H]
    \centering
    \includegraphics[scale=0.5]{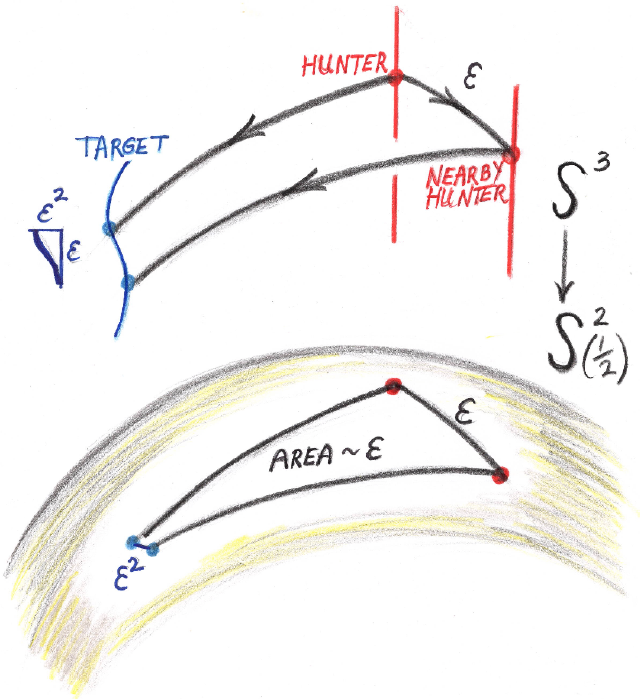}
    \caption{\centering More holonomy.}
    \label{fig:fig13}
\end{figure}

A point on the red hunter fiber shoots horizontally for the blue image target fiber. Then the hunter fiber moves sidewise horizontally by an amount $\epsilon$, and afterwards shoots again for the same target.

The two impact points on the blue image target fiber are separated by an amount proportional to $\epsilon$, as confirmed below. If the target fiber is within $\epsilon$ of vertical in the $C^1$ topology, then the separation between the two impact points is of order $\epsilon$ in the vertical direction, but only of order $\epsilon^2$ in the horizontal direction.

Projecting down to the two-sphere of radius $1/2$, we see an ``almost closed'' triangle with one side of length $\epsilon$, and at the other end a slight separation of order $\epsilon^2$. The triangle has area proportional to $\epsilon$, which confirms the separation between the two impact points up in $S^3$.

It follows that for each point on the hunter fiber $F$, say at location $t$, the covariant derivative $\nabla_U V(t)$ can be made arbitrarily close in the $C^0$ topology to the vector $-U$ for fixed target image fiber $F_0'$ and for hunters in a sufficiently small open subset of $S^2$. The understanding here is that to do this, we first make the blue image target fibers sufficiently close to vertical in the $C^\infty$ topology (though $C^1$ will do the job).

Now we prove that the covariant derivative $\nabla_U \overline{V}$ of the balance field $\overline{V}$ can be made arbitrarily close to the vector $-U$ in the $C^0$ topology for sufficiently small open subsets of $S^2$.

We start with a diffeomorphism $f$ of $S^3$ close to the identity, a Hopf fiber $F_0$ which will be the source of our target, the blue image fiber $f(F_0)$ which is our actual target, and a nearby Hopf fiber $F$ which is our hunter of the moment.

At each point of $F$ we have a horizontal vector $V(t)$ whose exponential image runs horizontally from that point to a point on the target fiber $f(F_0)$.

Then we let $U$ be a tangent vector to $S^2$ at the point $p(F)$, and lift $U$ to a projectable vector field of the same name along $F$, meaning that at each point of $F$, the vector $U$ there projects to the given tangent vector $U$ at the point $p(F)$ on $S^2$. Up in $S^3$, we have the $C^0$ approximation $\nabla_U V(t) \sim -U$. Next, we consider the average value of $V$ along $F$, $\overline{V} = \frac{1}{2\pi} \int\limits_F V(t) dt$, and take its covariant derivative with respect to $U$:
\begin{equation}
\nabla_U \overline{V} = \nabla_U \frac{1}{2\pi} \int\limits_F V(t) dt = \frac{1}{2\pi} \int\limits_F \nabla_U V(t) dt \sim \frac{1}{2\pi} \int\limits_F (-U) dt \sim -U
\end{equation}
as desired. This completes the proof of \autoref{covariant_prop}.
\end{proof}

With these lemmas in hand, we are ready to prove \autoref{hessian_prop}.

\begin{proof}[Proof of \autoref{hessian_prop}]
We will show that the Hessian of the energy function 
\begin{equation}
E(V \text{ on } F) = \frac{1}{2\pi} \int\limits_F \frac{1}{2} \|V(t)\|^2 dt
\end{equation}
is positive definite, and hence that $E$ is strongly convex for each choice of initial Hopf fiber $F_0$.

A coordinate-free expression for the Hessian is
\begin{equation}
(\mathrm{Hess}\, E)(U, W) = \langle \nabla_U (\nabla E), W \rangle.
\end{equation}
Therefore we have
\begin{equation}\label{eq_hessian_pos}
(\mathrm{Hess}\, E)(U, W) = \langle \nabla_U (\nabla E), W \rangle = \langle \nabla_U (-\overline{V}), W \rangle = -\langle \nabla_U \overline{V}, W \rangle \sim -\langle -U, W \rangle = \langle U, W \rangle,
\end{equation}
using \autoref{gradient_prop} and \autoref{covariant_prop}. Thus, $\mathrm{Hess}(E)$ is positive definite and the energy function $E$ is strongly convex for each choice of initial Hopf fiber $F_0$, as claimed.
\end{proof}





We finally build on these results to prove  \autoref{unique_min_prop}.

\begin{proof}[Proof of  \autoref{unique_min_prop}]
We will show that the unique minimum point $m$ of the energy function $E$ on an open subset of $S^2$ varies smoothly with the choice of initial Hopf fiber $F_0$.

In local coordinates $x_1, x_2$ on $S^2$, the Hessian of the energy function $E$ can be expressed in terms of first and second partial derivatives of $E$ together with the Christoffel symbols $\Gamma^k_{ij}$ of the Riemannian metric as
\begin{equation}
(\mathrm{Hess}\, E)\left(\frac{\partial}{\partial x_i}, \frac{\partial}{\partial x_j}\right) = \frac{\partial^2 E}{\partial x_i \partial x_j} - \Gamma^k_{ij} \frac{\partial E}{\partial x_k}
\end{equation}
where we understand summation on $k$.

If we use spherical coordinates on $S^2$, with singularities at the poles, then the Christoffel symbols are bounded outside neighborhoods of these poles. Hence choosing our coordinates so that the unique minimum point $m$ of the energy function $E$ lies on the equator, and staying very close to $m$ so that the first partials $\partial E/\partial x_1$ and $\partial E/\partial x_2$ are almost zero,  we see that positive definiteness of $\mathrm{Hess}(E)$ will imply positive definiteness of the matrix $(\partial^2 E/\partial x_i \partial x_j)$.

This in turn will let us use the Implicit Function Theorem to finish the argument, and to facilitate that, we will switch notation for local coordinates, replacing $x_1$ and $x_2$ by $x$ and $y$, and using subscript notation for partial derivatives. In particular, we replace the matrix $(\partial^2E/\partial{x_i}\partial{x_j})$ by the matrix
\begin{equation}
D^2E = 
\begin{pmatrix}
E_{xx} & E_{xy} \\
E_{yx} & E_{yy}
\end{pmatrix}.
\end{equation}
Then we apply the Implicit Function Theorem to the function
\begin{equation*}
\varphi_{u, v}(x, y) = \big(E_x, \; E_y\big)
\end{equation*}
to see that there will be a smooth function
\begin{equation}
\psi \colon (u, v) \mapsto (x, y) = (x(u, v), y(u, v))
\end{equation}
such that
\begin{equation}
E_x(u, v, x(u, v), y(u, v)) = 0 \qquad \text{and} \qquad E_y(u, v, x(u, v), y(u, v)) = 0.
\end{equation}
This will tell us that the location $(x(u, v), y(u, v))$ of the minimum of the energy function $E(u, v, x, y)$ depends smoothly on the parameters $(u, v)$.

Thus the minimum point $m$ of $E$ varies smoothly with the choice of initial Hopf fiber $F_0$, completing the proof of \autoref{unique_min_prop}.
\end{proof}

\subsection{Proof of the smooth slice theorem for $\diffst$}

Refer again to \autoref{fig:fig7}, where we start with a left coset $f \, \auth$ which is close to $\auth$, and have modified the representative $f$ so that it is close to the identity and moves points horizontally.


We must find a balanced diffeomorphism of $S^3$ in this same coset of $\auth$, and begin by focusing on a random fiber $F_0$ of the Hopf fibration $\HH$, and the corresponding blue image fiber $f(F_0)$, our target.

We then conduct our earlier described hunt, in which nearby Hopf fibers $F$ shoot horizontally for the fixed target $f(F_0)$ with a horizontal vector field $V$.

We consider the strongly convex energy function
$$E(V \text{ on } F) = \frac{1}{2\pi} \int\limits_F \frac{1}{2} \|V(t)\|^2 dt,$$
with the Hopf fiber $F$ sitting in $S^3$ directly above the unique minimum point $m$ of this energy function winning the hunt by shooting for the target with a balanced vector field.

By  \autoref{unique_min_prop}, this winning Hopf fiber $F$ depends smoothly on the choice of initial Hopf fiber $F_0$. The horizontal map from $F_0$ to $F$, when sufficiently close to the identity in the $C^\infty$ topology, will be a diffeomorphism, and thus an element of $\auth$.

Hence the balanced vector field $V$ from $F$ to $f(F_0)$ is smooth, and its Riemannian exponential
\begin{equation}
f_{\mathrm{bal}} = \Exp\, V
\end{equation}
lies in the same left coset of $\auth$ as did our starting diffeomorphism $f$ because both take the Hopf fibration to the same image fibration. In \autoref{fig:fig7}, the diffeomorphism $f_{\mathrm{bal}}$ takes the place of the question mark.

In this way, we have shown that the bundle $\auth \subset \diffst \to \mathrm{Fib}(S^3)$ admits a smooth slice consisting of balanced diffeomorphisms, and it then follows that it is a smooth fiber bundle of Fréchet manifolds, and that its base space $\mathrm{Fib}(S^3)$ is a smooth Fréchet manifold, modeled on the Fréchet space of balanced vector fields on $S^3$.

\section{Smooth Fr\'echet bundle structure}

The key to computing the homotopy type of $\Fib(S^3)$ is to leverage two related bundle structures, the ``big bundle'' from Theorem \ref{slice_thm},
\begin{equation}\label{big_bd}
 \aut \hookrightarrow \diff \xrightarrow{p}   \Fib(S^3)  
 \end{equation}
and a ``small bundle'',
\begin{equation}\label{small_bd}
 \vaut \hookrightarrow \aut \xrightarrow{p} \mathrm{Diff}^+(S^2),  
\end{equation}
 where the first map is the inclusion of the subgroup of automorphisms of the Hopf fibration that take each Hopf fiber to itself, with stretching and compressing allowed. Note that for any representative $f$ of a coset of $V\aut$ in $\aut$ the induced transformation $p(f)$ on the base space of the fiber bundle $H$ is the same.  This gives the map to the group of orientation preserving diffeomorphisms of the two sphere, $\mathrm{Diff}^+(S^2)$, which is also known to be a Fr\'echet manifold (see, for example, \cite{Ham}).



The last map in the sequence \eqref{small_bd} is easily seen to be onto $\mathrm{Diff}^+(S^2)$ because diffeomorphisms
of $S^2$ which are supported in small open sets can be easily lifted to automorphisms of $\HH$, and the family of such diffeomorphisms is known~\cite[Chapter $2$]{banyaga_book} to generate $\mathrm{Diff}^+(S^2)$.

To see that the last map in sequence \eqref{big_bd} is onto $\mathrm{Fib}(S^3)$, we need the following fact.

\begin{proposition}\label{prop_equivalent_fibrations}
Any two smooth fibrations of $S^3$ by simple closed curves are equivalent, in the sense that there is a diffeomorphism of $S^3$ taking one to the other.
\end{proposition}

\begin{proof}
Given any smooth fiber bundle $S^1 \hookrightarrow S^3 \to S^2$, we consider $\mathrm{Diff}(S^1)$ as its group, and note that two such bundles are smoothly equivalent if and only if their associated principal $\mathrm{Diff}(S^1)$ bundles are smoothly equivalent.

But the two $S^1$ bundles are known to be topologically equivalent, so their associated principal bundles are at least topologically equivalent as well. Then the theorem of Müller and Wockel \cite{smoothcontbundles} kicks in to say that for principal bundles structured by Fr\'echet Lie groups, smooth and continuous equivalences are
the same. So our two original smooth $S^1$ bundles must also be smoothly equivalent, as desired.
\end{proof}

\begin{rem}
In the topological category, \autoref{prop_equivalent_fibrations} follows immediately from the classification of $S^1$ bundles over $S^2$. It is also possible to adapt the proof from the topological to the smooth setting in this particular case.
\end{rem}

\newpage

\subsection{The homotopy type of the moduli space}

We are ready to prove  \autoref{mainthm} and \autoref{better_main_thm}, which determine the homotopy type and the homeomorphism type of the space of smooth fibrations of $S^3$ by simple closed curves.

We give the proof of \autoref{mainthm} in the case where the fibers are oriented; an analogous proof gives the conclusion when they are unoriented.

\begin{proof}[Proof of \autoref{mainthm}] We begin by comparing our two Fréchet fiber bundles with finite-dimensional sub-bundles, which will help us towards our goal of determining the homotopy type of the base space
$\mathrm{Fib}(S^3)$ of our big bundle in \eqref{big_bd}.

\begin{equation}\label{eq_small_diagram}
\begin{tikzcd}
\vaut \arrow[r, phantom,"\subset"]
&
\auth
\arrow[r, phantom,"\rightarrow"]
&
\mathrm{Diff}^+(S^2)
\\
S^1 \arrow[r, phantom,"\subset"]
\arrow[u, phantom, sloped, "\subset"]
&
U(2) \arrow[r, phantom,"\rightarrow"]
\arrow[u, phantom, sloped, "\subset"]
&
SO(3)
\arrow[u, phantom, sloped, "\subset"]
\end{tikzcd}
\end{equation}

On the bottom row of \eqref{eq_small_diagram}, the circle subgroup $S^1$ of the unitary group $U(2)$ consists of all rotations of $S^3$ which turn each Hopf circle within itself by the same amount, while $\mathrm{SO}(3)$ as usual denotes the group of orientation-preserving rigid motions of the two-sphere.

\begin{equation}\label{eq_big_diagram}
\begin{tikzcd}
\auth \arrow[r, phantom,"\subset"]
&
\diffst
\arrow[r, phantom,"\rightarrow"]
&
\mathrm{Fib}(S^3) 
\\
U(2) \arrow[r, phantom,"\subset"]
\arrow[u, phantom, sloped, "\subset"]
&
O(4) \arrow[r, phantom,"\rightarrow"]
\arrow[u, phantom, sloped, "\subset"]
&
S^2 \sqcup S^2
\arrow[u, phantom, sloped, "\subset"]
\end{tikzcd}
\end{equation}

On the bottom row of \eqref{eq_big_diagram}, the unitary subgroup $U(2) \subset O(4)$ consists of the rigid motions of the
$3$-sphere which permute the oriented fibers of the Hopf fibration $H$, while $S^2 \sqcup S^2$ denotes the set of all possible Hopf fibrations of $S^3$ by oriented parallel great circles. These rigid symmetries of the Hopf fibration are discussed in~\cite{gluckwarnerziller}.

Now we compare the infinite-dimensional bundles above with their finite-dimensional sub-bundles by comparing corresponding terms.

We focus first on the small bundles on lines \eqref{small_bd} and \eqref{eq_small_diagram}, so that we can learn the shape of
the total space $\aut$ on \eqref{small_bd}, which will then reappear as the fiber of the big bundle on line \eqref{big_bd}.

To begin, consider the inclusion of fibers, $S^1 \subset \vaut$.

It is well known that the space $\mathrm{Homeo}^+(S^1)$ of orientation-preserving homeomorphisms of the circle $S^1$ with the compact-open topology deformation retracts to its subspace $\mathrm{SO}(2)$ of rotations; see for example \cite[Prop. $4.2$]{ghysCircle}. The proof by straight line homotopy works in the differentiable case as well, so that the subset $\mathrm{Diff}^+(S^1)$ in the smooth topology
deformation retracts within itself to $\mathrm{SO}(2)$.

Applying this lemma simultaneously on each Hopf circle, we begin by deforming a diffeomorphism $f$ within $\vaut$ so that at every moment, it takes each Hopf circle within itself, and
at the end simply rotates each of these by some amount, which can vary from circle to circle. We note that the uniformity of this deformation process guarantees the smoothness required to
keep it within $\vaut$.

But every smooth map from $S^2$ to $S^1$ is smoothly null-homotopic, so a further deformation
of $f$ ends with a diffeomorphism of $S^3$ which rotates each Hopf circle by the same amount,
meaning that it lies in our circle subgroup $S^1$ of $U(2)$.
Hence the inclusion of fibers, $S^1 \subset \vaut$, is a homotopy equivalence.

We bring the proof of  \autoref{mainthm} to a close, as follows.

The inclusion \eqref{eq_small_diagram} of our small fiber bundles has homotopy equivalences of fibers as
noted above, and of base spaces by Smale's theorem \cite{Smale}, so that in the induced maps of
their long exact homotopy sequences, two-thirds of these maps are isomorphisms. It then follows from the Five Lemma that the remaining maps are also isomorphisms, and we learn
that the inclusion $U(2) \subset \aut$ of their total spaces is at least a weak homotopy
equivalence.

Then the inclusion of fiber bundles \eqref{eq_big_diagram} has a weak homotopy equivalence of fibers by what we just said above, and a homotopy equivalence of total spaces by Hatcher's
theorem~\cite{hatcher_smale}, so that once again applying the Five Lemma to their long exact homotopy
sequences, we learn that the inclusion $S^2 \sqcup S^2 \subset \mathrm{Fib}(S^3)$ of their base spaces is a weak
homotopy equivalence.

It remains to promote the weak homotopy equivalence $S^2 \sqcup S^2 \subset \mathrm{Fib}(S^3)$ to a homotopy
equivalence, and to do this we need to confirm that our moduli space, already known to be
a Fr\'echet manifold, is metrizable.

We know from  \autoref{separable_metrizable} that $\diffst$ is separable and metrizable. This in turn implies that $\diffst$ is second countable: we just take a countable dense subset and use each of its points as the center of open balls with rational radii.

Now the projection map $P \colon \diffst \to \mathrm{Fib}(S^3)$ is an open map. That's because, given an open subset $U$ of $\diffst$, its ``saturation'' $U\,\mathrm{Aut}(H)$ is also open. And then the projection of this to $\mathrm{Fib}(S^3)$ is open because $\mathrm{Fib}(S^3)$ has the quotient topology. Thus
a countable basis for the topology of $\diffst$ projects to a countable basis for the topology of $\mathrm{Fib}(S^3)$. Hence $\mathrm{Fib}(S^3)$ is second countable. But then by  \autoref{component_metrizable}, our moduli space $\mathrm{Fib}(S^3)$ is metrizable.

So the inclusion $S^2 \sqcup S^2 \subset \mathrm{Fib}(S^3)$ is a weak homotopy equivalence between metrizable Fr\'echet manifolds, and hence by  \autoref{homotopy_upgrade} it is in fact a homotopy equivalence. This completes the proof of \autoref{mainthm} in the case of oriented fibers, and similarly in the case of unoriented fibers.
\end{proof}

From this we deduce the even stronger \autoref{better_main_thm}.

\begin{proof}[Proof of \autoref{better_main_thm}] By \cite{anderson1966}, each infinite-dimensional separable Fréchet vector space is homeomorphic, albeit not linearly, to the unique separable Hilbert vector space, which in turn is homeomorphic to a countably infinite product of real lines.

By \cite{henderson_schori_1970}, two infinite-dimensional separable metrizable Fr\'echet manifolds are homeomorphic if and only if they have the same homotopy type.

Thus the homotopy equivalence $S^2 \sqcup S^2 \subset \mathrm{Fib}(S^3)$ leads to the desired homeomorphism
\begin{equation*} 
(S^2 \sqcup S^2) \times \mathbb{R}^{\infty} \simeq \mathrm{Fib}(S^3).
\end{equation*}

\end{proof}

\setlength{\parskip}{0pt}

\newpage

\bibliographystyle{amsalpha}
\bibliography{biblioSF}

\end{document}